\documentclass{article}
\usepackage{amssymb,amsfonts,amsmath,amsthm}
\usepackage{hyperref}
\usepackage{parskip}
\usepackage{setspace}
\usepackage{fancyhdr}
\usepackage[dvipsnames]{xcolor}
\usepackage[numbers]{natbib}
\usepackage[bottom]{footmisc}
\usepackage[utf8]{inputenc}
\usepackage[english]{babel}
\usepackage{xcolor}
\usepackage{geometry}
\usepackage{graphicx}
\newcommand\myshade{85}
\colorlet{mylinkcolor}{blue}
\colorlet{mycitecolor}{red}
\colorlet{myurlcolor}{Aquamarine}
\hypersetup{
	linkcolor  = mylinkcolor,
	citecolor  = mycitecolor,
	urlcolor   = myurlcolor!\myshade!black,
	colorlinks = true,
}
\title{On some generalized number theoretic functions and Ighachanea-Akkouchia Holder's inequalities}

\author{Omprakash Atale${}^{1}$\footnote{${}^{1}$ Department of Mathematics, 
Khandesh College Education Society's. Moolji Jaitha College, India.}} 
\date{July 1, 2022}

\begin{document}

\maketitle

\begin{abstract}
Recently, it has been shown by Ighachanea and Akkouchia \cite{0.1} that using binomial coefficients, one can derive some new refinements of Holder's inequalities. This inequalities then can be applied to a wide class of special functions such as the Nielsen's beta function and some extended gamma functions. In this paper, we have derived some generalizations of previously known number theoretic functions. Furthermore, based on the results of Ighachanea and Akkouchia, Holder's inequalities for the derived generalized functions are established. 
\end{abstract}
\textbf{keywords:} Number theoretic functions, Holder's inequalities, Nielsen's beta function, Extended gamma function.

\tableofcontents

\newtheorem{theorem}{Theorem}[section]
\newtheorem{corollary}{Corollary}[theorem]
\newtheorem{lemma}[theorem]{Lemma}
\theoremstyle{definition}
\newtheorem{definition}{Definition}[section]
\renewcommand\qedsymbol{$\blacksquare$}

\section{Introduction}
Holder's inequalities have played an important role in wide application based areas of mathematics. In this paper, we have applied the generalizations of Holder's inequality given by Ighachanea and Akkouchia in [1] to a wider class of some special number theoretic functions which are also derived in this paper. The paper is arranged as follows. In this section, we have introduced some extended definitions of the gamma function which we will be subsequently using throughout the paper. In section 2, we present Holder's inequality and some preliminary results. In section 3, the results of Ighachanea and Akkouchia are presented that we will be applying throughout our paper to various special functions. In section 4, we have derived a analogue of Nielsen's beta function and derived some of its properties. In section 5, we have presented and extension of the Chaudhary-Zubair gamma function and derived some of its properties. The next sections are followed by applying inequalities from section 3 to the special functions derived in previous two sections. 

The \textit{p-k} gamma function or two parameter gamma functions is a parametric deformation of the classic gamma function given by:\footnote{$ _{p}\Gamma_{k}(x) \Rightarrow\: _{k}\Gamma_{k}(x) = \Gamma_{k}(x)  $ as $ p = k $ and $ _{p}\Gamma_{k}(x) \Rightarrow \:_{1}\Gamma_{1}(x)= \Gamma (x)  $ as $ p,k\rightarrow 1 $.}  
\begin{definition}{[\cite{1}, pg. 3]}(\textit{p-k} Gamma Function)
Given $ x\in C / kZ^{-}; k,p\in R^{+}-\lbrace0\rbrace$ and $\Re(x)>0,$
then the integral representation of \textit{p-k} Gamma Function is given by
\begin{equation*}
_{p}\Gamma_{k}(x)=\int^{\infty}_{0}e^{-\frac{t^{k}}{p}}t^{x-1}dt.\tag{1.1}
\end{equation*}
\end{definition}
The above definition reduces to $k$-gamma function when $p=k$ \cite{2}. $\Gamma_{k}(x)$ appears in a variety of contexts, such as, the combinatorics of creation and annihilation operators \cite{3}, \cite{4} and the perturbative computation of Feynman integrals,
see \cite{5}. For more applications of $k$-gamma function refer to \cite{6}-\cite{19}
\begin{definition}{[\cite{1}, pg. 5]} For $ x\in C/kZ^{-} ;  k,p \in R^{+}-\lbrace 0 \rbrace $ and $\Re(x)>0, n\in N. $ The fundamental equations satisfied by \textit{p-k} Gamma Function, $ _{p}\Gamma_{k}(x) $ are,
\begin{equation*}
_{p}\Gamma_{k}(p)=\frac{p^{\frac{p}{k}}}{k}\Gamma(\frac{p}{k}).\tag{1.3}
\end{equation*}
\begin{equation*}
_{p}\Gamma_{k}(x)\:_{p}\Gamma_{k}(-x)=\frac{\pi}{xk}\:\frac{1}{\sin (\frac{\pi x}{k})}.\tag{1.4}
\end{equation*}
\begin{equation*}
_{p}\Gamma_{k}(x)\:_{p}\Gamma_{k}(k-x)=\frac{p}{k^{2}}\:\frac{\pi}{\sin (\frac{\pi x}{k})}.\label{1.5}\tag{1.5}
\end{equation*}
\begin{equation*}
\:\prod_{{0 \le r \le m - 1}}^{}\:{}_{p}\Gamma_{k}(x+\frac{kr}{m})=\frac{p^{\frac{m-1}{2}}}{k^{m-1}}(2 \pi )^{\frac{(m-1)}{2}}m^{\frac{1}{2}-\frac{mx}{k}}\:_{p}\Gamma_{k}(mx); m=2,3,4,... \: .\tag{1.6}
\end{equation*}
Relation of $p$-$k$ gamma function with $k$ gamma function and the classic gamma function is given by
\begin{equation*}
_p{\Gamma _k}\left( x \right) = {\left( {\frac{p}{k}} \right)^{\frac{x}{k}}}{\Gamma _k}\left( x \right) = \frac{{{p^{\frac{x}{k}}}}}{k}\Gamma \left( {\frac{x}{k}} \right).\label{1.7}\tag{1.7}
\end{equation*}
\end{definition}
We kindly request readers to make themselves familiar with the $k$-gamma function introduced in \cite{2}. Further generalizations of the \textit{k-}gamma function and ordinary gamma function can be found in \cite{25}-\cite{27}.
\section{Holder's inequalities and some preliminary results}
\begin{theorem}[Holder's Inequality]
Let $\left( {\Omega ,\mathcal{F},\mu } \right)$ be a measure space where $\mu$ is a positive measure. Let $\xi,\Tilde{\xi}:\Omega\to{\textbf{C}}$ be two measurable functions. Then, for all $p,q\geq 1$ and $p^{-1}+q^{-1}=1$, we have 
\begin{equation*}
\int\limits_\Omega  {\left| {\xi \tilde \xi } \right|} d\mu \left( t \right) \le {\left( {{{\int\limits_\Omega  {\left| \xi  \right|} }^p}d\mu \left( t \right)} \right)^{\frac{1}{p}}}{\left( {{{\int\limits_\Omega  {\left| {\tilde \xi } \right|} }^p}d\mu \left( t \right)} \right)^{\frac{1}{q}}}\label{2.1}\tag{2.1}    
\end{equation*}
\end{theorem}
From \cite{0.2}, we have the following two theorems
\begin{theorem}
Let $n$ and $m$ be two integers and let $a_{i}\in\mathbb{R}^{+}$. Set $i_0:=m$, $i_n:=0$ and 
	$A:=\{(i_1,\ldots,i_{n-1}): 0\leq i_k \leq i_{k-1},\; 1\leq k \leq n-1\}.$
	Then, we have 
	\begin{align*}
	\left(\sum_{k=1}^{n}\nu_{k}a_{k}\right)^m=\sum_{(i_1,\ldots,i_{n-1})\in A}C_{A}\nu_{1}^{i_{0}-i_{1}}\nu_{2}^{i_{1}-i_{2}}\ldots \nu_{n}^{i_{n-1}-i_{n}}a_{1}^{i_{0}-i_{1}}a_{2}^{i_{1}-i_{2}}\ldots a_{n}^{i_{n-1}-i_{n}},\label{2.2}\tag{2.2}
	\end{align*}
	where, $C_{A}=\binom{i_{1}}{i_{2}}\ldots\binom{i_{n-2}}{i_{n-1}},$ the $\binom{i_{k-1}}{i_{k}}$ is the binomial coefficient.\\
\end{theorem}
\begin{theorem}
For $k=1,2,\dots,n,$ let $a_{k}\geq0$ and let  $\nu_{k}>0$ satisfy $\sum_{k=1}^{n}\nu_{k} =1.$ Then for all integers $m\geq1$, we have 
    \begin{align*}
	\prod_{k=1}^{n}a_{k}^{\nu_{k}}+r^{m}_0\left(\sum_{k=1}^na_k-n\sqrt[n]{\prod_{k=1}^{n}a_{k}}\right) \leq\Big{(} \sum_{k=1}^{n}\nu_{k}a_{k}^{\frac{1}{m}}\Big{)}^{m}
	\leq \sum_{k=1}^{n}\nu_{k}a_{k},\label{2.3}\tag{2.3}
	\end{align*}
	where $r_0=\min\{\nu_k:k=1,\ldots, n\}$.
	
	Moreover, if we set $ U_{m}:=\Big{(} \sum_{k=1}^{n}\nu_{k}a_{k}^{\frac{1}{m}}\Big{)}^{m},$ then $\{U_{m}\}$ is a decreasing sequence and we have$\lim_{m\rightarrow\infty} U_{m}=\prod_{k=1}^{n}a_{k}^{\nu_{k}}.$
\end{theorem}
\section{Ighachanea-Akkouchia Holder's inequalities}
Using theorems 2.3 and 2.4, Ighachanea and Akkouchia \cite{0.1} derived the following refinements of the Holder's inequality.
\begin{theorem}[Ighachanea-Akkouchia inequality type-I]
Let $\left( {\Omega ,\mathcal{F},\mu } \right)$ be a measure space where $\mu$ is a positive measure. Let $n$ be a positive integer and	let $\xi_{1},\xi_{2},\ldots,\xi_{n}$ be $\mu$-measurable functions such that $\xi_{k}\in\mathcal{L}^{p_{k}}(\mu)$, for all $k=1,\ldots,n.$  Then for all integers $m\geq2$, the inequalities 
		\begin{align*}
		\int_{\Omega}^{}\prod_{k=1}^{n}|\xi_{k}(t)|d\mu(t)&nr_{0}^{m}\prod_{k=1}^{n}||\xi_{k}||_{p_{k}}\Big{(}1-\prod_{k=1}^{n}||\xi_{k}||_{p_{k}}^{\frac{-p_{k}}{n}}\int_{\Omega}^{}\prod_{k=1}^{n}|\xi_{k}(t)|^{\frac{p_{k}}{n}}d\mu(t)\Big{)}\\&
			\hspace{-3cm}\leq\sum_{(i_1,\ldots,i_{n-1})\in A} C_{A} \frac{1}{p_{1}^{i_{0}-i_{1}}..p_{n}^{i_{n-1}-i_{n}}}\prod_{k=1}^{n}||\xi_{k}||_{p_{k}}^{1-\frac{p_{k}(i_{k}-i_{k-1})}{m}}\int_{\Omega}^{}\prod_{k=1}^{n}|\xi_{k}(t)|^{\frac{p_{k}(i_{k}-i_{k-1})}{m}}d\mu(t)\leq\prod_{k=1}^{n}||\xi_{k}||_{p_{k}},\label{3.1}\tag{3.1}
		\end{align*}
holds for $p_{k}>1,$ such that $\sum_{k=1}^{n}\frac{1}{p_{k}}=1,$ 	where $r_0=\min\{\frac{1}{p_k}:k=1,\ldots, n\}$.\\
	\end{theorem}
\begin{theorem}[Ighachanea-Akkouchia inequality type-II]
		\label{th2.2}
		Let $n,N$ be two integers and $\{Q_{j,k}\}\subset \mathbb{R},$ where
		 $k=1,2,\ldots,n$ and $j=1,2,\ldots,N.$ Let $p_{k}>1,$ such that $\sum_{k=1}^{n}\frac{1}{p_{k}} =1,$  Then the inequalities
	\begin{align*}
		\sum_{j=1}^{N}\Big{|}\prod_{k=1}^{n}Q_{j,k}\Big{|}&nr_{0}^{m}\prod_{k=1}^{n}\Big{(}\sum_{j=1}^{N}|Q_{j,k}|^{p_{k}}\Big{)}^{\frac{1}{p_{k}}}\Big{(}1-\prod_{k=1}^{n}\Big{(}\sum_{j=1}^{N}|Q_{j,k}|^{p_{k}}\Big{)}^{\frac{-1}{n}}\sum_{j=1}^{N}\prod_{k=1}^{n}|Q_{j,k}|^{\frac{p_{k}}{n}}\Big{)}\\&
		\hspace{-2.5cm}\leq\sum_{(i_1,\ldots,i_{n-1})\in A} C_{A} \frac{1}{p_{1}^{i_{0}-i_{1}}..p_{n}^{i_{n-1}-i_{n}}}\prod_{k=1}^{n}\Big{(}\sum_{j=1}^{N}|Q_{j,k}|^{p_{k}}\Big{)}^{\frac{1}{p_{k}}-\frac{(i_{k}-i_{k-1})}{m}}\sum_{j=1}^{N}\prod_{k=1}^{n}|Q_{j,i}|^{\frac{p_{k}(i_{k}-i_{k-1})}{m}}\\&
		&&\hspace{-2.5cm}\leq\prod_{k=1}^{n}\Big{(}\sum_{j=1}^{N}|Q_{j,k}|^{p_{k}}\Big{)}^{\frac{1}{p_{k}}},\label{3.2}\tag{3.2}
	\end{align*}
	
is valid, where $r_0=\min\{\frac{1}{p_k}:k=1,\ldots, n\}$.
\end{theorem}
The detailed proof of above two theorems can be found in [1].
\section{On some analogues of Nielsen's beta function}
\subsection{Basic properties}
\begin{definition} 
For $x>0$, we define Nielsen's $\beta$ function as follows\footnote{\cite{14}-\cite{17}}
\begin{equation*}
\beta \left( x \right)  = \int\limits_0^1 {\frac{{{t^{x - 1}}}}{{1 + t}}dt}  = \int\limits_0^\infty  {\frac{{{e^{ - xt}}}}{{1 + {e^{ - t}}}}dt} = \sum\limits_{k = 0}^\infty  {\frac{{{{\left( { - 1} \right)}^k}}}{{k + x}} = \frac{1}{2}} \left\{ {\psi \left( {\frac{{x + 1}}{2}} \right) - \psi \left( {\frac{x}{2}} \right)} \right\}.\label{4.1}\tag{4.1}    
\end{equation*}
where $\psi \left( x \right) = \frac{d}{{dx}}\log \Gamma \left( x \right)$.
\end{definition}
Nielsen's $\beta$ function satisfies the following properties
\begin{equation*}
\beta \left( {x + 1} \right) = \frac{1}{x} - \beta \left( x \right),\label{4.2}\tag{4.2} 
\end{equation*}
\begin{equation*}
\beta \left( x \right) + \beta \left( {1 - x} \right) = \frac{\pi }{{\sin \pi x}}.\label{4.3}\tag{4.3}    
\end{equation*}
Further additional properties can be found in \cite{18}. From \cite{1}, we have 
\begin{equation*}
\frac{1}{{_p{\Gamma _k}}(x)} = \frac{x}{{k{p^{\frac{x}{k}}}}}{e^{\frac{x}{k}\gamma }}\prod\limits_{n = 1}^\infty  {\left[ {\left( {1 + \frac{x}{{nk}}} \right){e^{ - \frac{x}{{nk}}}}} \right]}.\label{4.4}\tag{4.4}     
\end{equation*}
From, Eqn. (\ref{4.4}), we get the following value for $p$-$k$ digamma function
\begin{align*}
{}_{p}{\psi _k}\left( x \right) = \frac{d}{{dx}}\ln \left( {_p{\Gamma _k}\left( x \right)} \right) &= \frac{{\ln p - \gamma }}{k} - \frac{1}{x} + \sum\limits_{n = 1}^\infty  {\frac{x}{{nk(nk + x)}}}\label{4.5}\tag{4.5} \\& = \frac{{\ln p - \gamma }}{k} + \sum\limits_{n = 0}^\infty  {\left( {\frac{1}{{nk + k}} - \frac{1}{{nk + x}}} \right)}.\label{4.6}\tag{4.6}     
\end{align*}
\begin{theorem}
For $x, y>0$ and $p, k\in\mathbb{R}^{+}$, we have 
\begin{equation*}
{}_{p}{\psi _k}\left( x \right){ - _p}{\psi _k}\left( y \right) = {\psi _k}\left( x \right) - {\psi _k}\left( y \right)\label{4.7}\tag{4.7}
\end{equation*}
where ${}_{p}{\psi _k}\left( x \right)$ is the 
$p$-$k$ digamma function from Eqn. (\ref{4.5}) and ${\psi _k}\left( x \right)$ is the $k$-digamma function\footnote{$k$-digamma function can be obtained from Eqn. (\ref{4.5}) by letting $p=k$.} \cite{19}.
\end{theorem}
\begin{proof}
Proof follows from the definition of $p$-$k$ digamma function and $k$-digamma function.
\end{proof}
\begin{definition}
The $p$-$k$ extension of the Nielsen's $\beta$ function ${}_{p}\beta_{k}(x)$ for $x>0$ is defined as
\begin{align*}
{}_{p}{\beta _k}\left( x \right) &= \frac{p}{2}\left\{ {_p{\psi _k}\left( {\frac{{x + k}}{2}} \right){ - _p}{\psi _k}\left( {\frac{x}{2}} \right)} \right\}\label{4.8}\tag{4.8}\\& = \frac{p}{k}\sum\limits_{n = 0}^\infty  {\left( {\frac{k}{{2nk + x}} - \frac{k}{{2nk + x + k}}} \right)}\label{4.9}\tag{4.9}\\& = \frac{p}{k}\int\limits_0^\infty  {\frac{{{e^{ - \frac{{xt}}{k}}}}}{{1 + {e^{ - t}}}}dt}\label{4.10}\tag{4.10}\\&= \frac{p}{k}\int\limits_0^1 {\frac{{{t^{\frac{x}{k} - 1}}}}{{1 + t}}dt}\label{4.11}\tag{4.11},
\end{align*}
where ${}_{p}{\beta _k}\left( x \right)={\beta _k}\left( x \right)$ when $p=k$ and ${}_{p}{\beta _k}\left( x \right)={\beta}\left( x \right)$ when $p=k=1$.
\end{definition}
\begin{theorem}
${}_{p}{\beta _k}\left( x \right)$ satisfies the functional equation
\begin{equation*}
{}_{p}{\beta _k}\left( {x + k} \right) = \frac{p}{x} - {}_{p}{\beta _k}\left( x \right)\label{4.12}\tag{4.12}  \end{equation*}
and the reflection formula
\begin{equation*}
{}_{p}{\beta _k}\left( x \right){ + _p}{\beta _k}\left( {k - x} \right) = \frac{p^{2}}{k^{2}}\frac{\pi }{{\sin \frac{{\pi x}}{k}}}.\label{4.13}\tag{4.13}  
\end{equation*}
\end{theorem}
\begin{proof}
From Eqn. (\ref{4.11}), we have
\begin{equation*}
{}_{p}{\beta _k}\left( {x + k} \right){ + _p}{\beta _k}\left( x \right) = \frac{p}{k}\int\limits_0^1 {\frac{{{t^{\frac{x}{k}}} + {t^{\frac{x}{k} - 1}}}}{{1 + t}}dt = \frac{p}{k}\int\limits_0^1 {{t^{\frac{x}{k} - 1}}dt = } \frac{p}{x}}.\label{4.14}\tag{4.14}    
\end{equation*}
Now,
\begin{equation*}
{}_{p}{\beta _k}\left( x \right){ + _p}{\beta _k}\left( {k - x} \right) = \frac{p}{2}\left\{ {_p{\psi _k}\left( {\frac{x}{2}+\frac{k}{2}} \right){ - _p}{\psi _k}\left( {\frac{x}{2}} \right){ + _p}{\psi _k}\left( {k - \frac{x}{2}} \right){ - _p}{\psi _k}\left( {\frac{k}{2} - \frac{x}{2}} \right)} \right\},\label{4.15}\tag{4.15}    
\end{equation*} 
\begin{equation*}
= \frac{p}{2}\left\{ {_p{\psi _k}\left( {k - \left( {\frac{k}{2} - \frac{x}{2}} \right)} \right){ - _p}{\psi _k}\left( {\frac{k}{2} - \frac{x}{2}} \right){ + _p}{\psi _k}\left( {k - \frac{x}{2}} \right){ - _p}{\psi _k}\left( {\frac{x}{2}} \right)} \right\}.\label{4.16}\tag{4.16}  \end{equation*}
Now, logarithmically differentiating Eqn. (\ref{1.5}), we get
\begin{equation*}
{}_{p}{\psi _k}\left( x \right){ + _p}{\psi _k}\left( {k - x} \right) = \frac{p}{{{k^2}}}\pi \cot \frac{{\pi x}}{k}.\label{4.17}\tag{4.17}    
\end{equation*}
Using the above relation in Eqn. (\ref{4.16}), we get
\begin{align*}
{}_{p}{\beta _k}\left( x \right){ + _p}{\beta _k}\left( {k - x} \right)&= \frac{p}{2}\left\{ {\frac{p}{{{k^2}}}\pi \cot \frac{\pi }{k}\left( {\frac{k}{2} - \frac{x}{2}} \right) + \frac{p}{{{k^2}}}\pi \cot \frac{{\pi x}}{{2k}}} \right\}\label{4.18}\tag{4.18}\\& = \frac{{{p^2}\pi }}{{2{k^2}}}\left\{ {\cot \left( {\frac{\pi }{2} - \frac{{\pi x}}{{2k}}} \right) + \cot \frac{{\pi x}}{{2k}}} \right\}\label{4.19}\tag{4.19} \\&= \frac{{{p^2}\pi }}{{2{k^2}}}\left\{ {\tan \left( {\frac{{\pi x}}{{2k}}} \right) + \cot \frac{{\pi x}}{{2k}}} \right\} \label{4.20}\tag{4.20}\\&= \frac{{{p^2}}}{{{k^2}}}\frac{\pi }{{2\cos \left( {\frac{{\pi x}}{{2k}}} \right)\sin \left( {\frac{{\pi x}}{{2k}}} \right)}}\label{4.21}\tag{4.21}\\& = \frac{{{p^2}}}{{{k^2}}}\frac{\pi }{{\sin \left( {\frac{{\pi x}}{k}} \right)}}.\label{4.22}\tag{4.22}       
\end{align*}
This completes our proof.
\end{proof}
\begin{definition}
For $x>0$, $p, k\in\mathbb{R}^{+}$ and $n\in\mathbb{N}$, we have
\begin{align*}
{}_{p}\beta _k^{\left( n \right)}\left( x \right) &= \frac{p}{{{2^{n + 1}}}}\left\{ {_p\psi _k^{\left( n \right)}\left( {\frac{{x + k}}{2}} \right){ - _p}\psi _k^{\left( n \right)}\left( {\frac{x}{2}} \right)} \right\}\label{4.23}\tag{4.23}\\& = \frac{{{{\left( { - 1} \right)}^n}p}}{{{k^{n + 1}}}}\int\limits_0^\infty  {\frac{{{t^n}{e^{ - \frac{{xt}}{k}}}}}{{1 + {e^{ - t}}}}dt}\label{4.24}\tag{4.24}\\& = \frac{p}{{{k^{n + 1}}}}\int\limits_0^1 {\frac{{{{\left( {\ln t} \right)}^n}{t^{\frac{x}{k} - 1}}}}{{1 + t}}dt},\label{4.25}\tag{4.25}
\end{align*}
\begin{equation*}
{}_{p}\beta _k^{\left( n \right)}\left( {x + k} \right) = {\left( { - 1} \right)^n}\frac{{n!p}}{{{x^{n + 1}}}}{ - _p}\beta _k^{\left( n \right)}\left( x \right).\label{4.26}\tag{4.26}    
\end{equation*}
\end{definition}
\begin{theorem}
i) ${}_{p}{\beta _k}\left( x \right)$ is positive and decreasing.\newline
ii) ${}_{p}\beta _k^{\left( n \right)}\left( x \right)$ is positive and decreasing when $n$ is an even integer.\newline
iii)${}_{p}\beta _k^{\left( n \right)}\left( x \right)$ is negative and increasing if $n$ is an odd integer.
\end{theorem}
\begin{proof}
Proof trivially follows from Eqn. (\ref{4.24}).
\end{proof}
\begin{theorem}
i) ${}_{p}{\beta _k}\left( x \right)$ is logarihmically convex on $\left( {0,\infty } \right)$.\newline
ii) ${}_{p}{\beta _k}\left( x \right)$ is completely monotonic on $\left( {0,\infty } \right)$.
\end{theorem}
\begin{proof}
\textit{i)} Let $r, s>1$, $\frac{1}{r} + \frac{1}{s} = 1$ and $x,y \in \left( {0,\infty } \right)$, then, using Eqn. (\ref{4.11}) and Hölder's inequality, we have
\begin{align*}
\bigg[\frac{k}{p}{}_{p}{\beta _k}\left( {\frac{x}{r} + \frac{y}{s}} \right)\bigg] &= \int\limits_0^1 {\frac{{{t^{\frac{x}{{kr}} + \frac{y}{{ks}} - 1}}}}{{1 + t}}dt} \label{4.27}\tag{4.27}\\& = \int\limits_0^1 {\frac{{{t^{\frac{{x - k}}{{kr}}}}}}{{{{\left( {1 + t} \right)}^{\frac{1}{r}}}}}\frac{{{t^{\frac{{y - k}}{{ks}}}}}}{{{{\left( {1 + t} \right)}^{\frac{1}{s}}}}}dt}\label{4.28}\tag{4.28} \\& \le {\left( {\int\limits_0^1 {\frac{{{t^{\frac{x}{k} - 1}}}}{{1 + t}}dt} } \right)^{\frac{1}{r}}}{\left( {\int\limits_0^1 {\frac{{{t^{\frac{y}{k} - 1}}}}{{1 + t}}dt} } \right)^{\frac{1}{s}}} \label{4.29}\tag{4.29}\\&= {\left[ {\frac{k}{p}{{}_{p}}{\beta _k}\left( x \right)} \right]^{\frac{1}{r}}}{\left[ {\frac{k}{p}{{}_{p}}{\beta _k}\left( y \right)} \right]^{\frac{1}{s}}}. \label{4.30}\tag{4.30}    
\end{align*}
\textit{ii)} Using Eqn. (\ref{4.24}), we have 
\begin{equation*}
{\left( { - 1} \right)^n}_p\beta _k^{\left( n \right)} = \frac{{{{\left( { - 1} \right)}^{2n}}p}}{{{k^{n + 1}}}}\int\limits_0^\infty  {\frac{{{t^n}{e^{ - \frac{{xt}}{k}}}}}{{1 + {e^{ - t}}}}dt}  \ge 0.\label{4.31}\tag{4.31}
\end{equation*}
Therefore,  ${}_{p}{\beta _k}\left( x \right)$ is completely monotonic on $\left( {0,\infty } \right)$. This completes our proof.
\end{proof}
\subsection{On the equivalent conditions for log-convexity}
If $f$ is any function differentiable over an interval and is logarithmically convex, then the function satisfies the following two inequalities:\newline
\textit{i)} For $x, y>0$, we have
\begin{equation*}
\log f\left( x \right) \ge \log f\left( y \right) + \frac{{f'\left( y \right)}}{{f\left( y \right)}}\left( {x - y} \right)\label{4.32}\tag{4.32}  
\end{equation*}
which equivalently can be written as 
\begin{equation*}
{\left( {\frac{{f\left( x \right)}}{{f\left( y \right)}}} \right)^{\frac{1}{{x - y}}}} \ge \exp \left( {\frac{{f'\left( y \right)}}{{f\left( y \right)}}} \right).\label{4.33}\tag{4.33}    
\end{equation*}
\textit{ii)} For $x>0$, we have 
\begin{equation*}
f''\left( x \right)f\left( x \right) \ge f'{\left( x \right)^2}.\label{4.34}\tag{4.34}    
\end{equation*}
Therefore, we obtain the following theorem.
\begin{theorem}
For $x, y>0$ and $p, k\in\mathbb{R}^{+}$, we have the following inequalities\newline
i)
\begin{equation*}
{\left( {\frac{{_p{\beta _k}\left( x \right)}}{{_p{\beta _k}\left( y \right)}}} \right)^{\frac{1}{{x - y}}}} \ge \exp \left( {\frac{{_p{{\beta '}_k}\left( y \right)}}{{_p{\beta _k}\left( y \right)}}} \right)\label{4.35}\tag{4.35}     
\end{equation*}
ii)
\begin{equation*}
{}_{p}{\beta _k}^{\prime \prime }{\left( x \right)_p}{\beta _k}\left( x \right){ \ge _p}{\beta _k}^\prime {\left( y \right)^2}.\label{4.36}\tag{4.36}     
\end{equation*}
\end{theorem}
\begin{proof}
Substitute $f(x)$ with ${}_{p}{\beta _k}\left( x \right)$ is Eqn. (\ref{4.33}) and (\ref{4.34}) and the desire result readily follows.
\end{proof}
\textbf{Proposition 4.1.} The following relation holds true
\begin{equation*}
{}_{p}{\beta _k}\left( x \right) = \frac{p}{k}{\beta _k}\left( x \right) = p\beta \left( {\frac{x}{k}} \right)\label{4.37}\tag{4.37}    
\end{equation*}
where ${\beta _k}\left( x \right)$ is the $k$-extension of Nielsen's beta function as introduced in \cite{20}.
\begin{proof}
The first equality follows from Eqn. (\ref{4.8}), (\ref{4.7}) and Definition 2.2  from \cite{20}. Similarly, the second inequality follows from Definition 2.2 from \cite{20} and the 4th equality of Eqn. (\ref{4.1}). \newline
Another way of proving the above proposition is using taking the counter examples. Authors in \cite{20} have proved the following inequality for ${\beta _k}\left( x \right)$ for $x, y \in [0,\infty )$
\begin{equation*}
{\beta _k}\left( {x + k} \right){\beta _k}\left( {y + k} \right) \le \ln 2{\beta _k}\left( {x + y + k} \right).\label{4.38}\tag{4.38}    
\end{equation*}
Therefore, if the second equality in Eqn. (\ref{4.37}) is true, then 
\begin{equation*}
k\beta \left( {\frac{x}{k} + 1} \right){\beta }\left( {\frac{y}{k} + 1} \right) \le \ln 2{\beta }\left( {\frac{x}{k} + \frac{y}{k} + 1} \right)\label{4.39}\tag{4.39}    
\end{equation*}
must also be true. By taking various counter examples, it turns that the above inequality holds true, therefore, we can conclude that the second equality in Eqn. (\ref{4.37}) is true. Similarly, one can prove the first equality by substituting $f(x)$ with ${}_{p}{\beta _k}\left( x \right)$ in Eqn. (\ref{4.32}) or by using Eqn. (\ref{4.35}) and further taking various counter examples.
\end{proof}
\begin{theorem} For $x,y \in [0,\infty )$ and  $p, k\in\mathbb{R}^{+}$, we have
\begin{equation*}
{}_{p}{\beta _k}{\left( {x + k} \right)}{}_{p}{\beta _k}\left( {y + k} \right) \le \frac{p\ln 2}{k}{}_{p}{\beta _k}\left( {x + y + k} \right)\label{4.40}\tag{4.40}    
\end{equation*}
\end{theorem}
\begin{proof}
Multiply Eqn. (\ref{4.38}) with $p$/$k$ twice and use relation \ref{4.37} to arrive at the desire result.
\end{proof}
Authors in \cite{20} have established the following two results for the $k$-extension of the Nielsen's beta function valid for $k>0$:\newline
\textit{i)} For $x, y, z\in\mathbb{R}^{+}$, we have 
\begin{equation*}
{\beta _k}\left( x \right){\beta _k}\left( {x + y + z} \right) - {\beta _k}\left( {x + y} \right){\beta _k}\left( {x + z} \right) > 0\label{4.41}\tag{4.41}    
\end{equation*}
\textit{ii)} For $a\ge{1}$ and $x\in[0,1]$, we have the following inequality which reverses if $0 < a \le 1$
\begin{equation*}
\frac{{{{\left[ {{\beta _k}\left( {1 + k} \right)} \right]}^a}}}{{{\beta _k}\left( {a + k} \right)}} \le \frac{{{{\left[ {{\beta _k}\left( {x + k} \right)} \right]}^a}}}{{{\beta _k}\left( {ax + k} \right)}} \le {\left( {\ln 2} \right)^{a - 1}}.\label{4.42}\tag{4.42}    
\end{equation*}
From the above two results, we can deduce the following theorems respectively.
\begin{theorem}
For $x, y, z\in\mathbb{R}^{+}$ and $p, k\in\mathbb{R}^{+}$, we have
\begin{equation*}
{}_{p}{\beta _k}{\left( x \right)_p}{\beta _k}\left( {x + y + z} \right){ - _p}{\beta _k}{\left( {x + y} \right)_p}{\beta _k}\left( {x + z} \right) > 0\label{4.43}\tag{4.43}.    
\end{equation*}
\end{theorem}
\begin{theorem}
For $a\ge{1}$, $x\in[0,1]$ and $p, k\in\mathbb{R}^{+}$, we have the following inequality which reverses for $0 < a \le 1$
\begin{equation*}
\frac{{{{\left[ {_p{\beta _k}\left( {1 + k} \right)} \right]}^a}}}{{_p{\beta _k}\left( {a + k} \right)}} \le \frac{{{{\left[ {_p{\beta _k}\left( {x + k} \right)} \right]}^a}}}{{_p{\beta _k}\left( {ax + k} \right)}} \le \frac{{{p^{a - 1}}}}{{{k^{a - 1}}}}{\left( {\ln 2} \right)^{a - 1}}.\label{4.44}\tag{4.44}    
\end{equation*}
\end{theorem}
\subsection{Some additional results for the n-th order }
For $n\in\mathbb{N}_{0}$, we define
\begin{equation*}
\left| {_p\beta _k^{\left( n \right)}\left( x \right)} \right| = {\left( { - 1} \right)^n}_p\beta _k^{\left( n \right)}\left( x \right)\label{4.45}\tag{4.45}    
\end{equation*}
which is decreasing for all $n\in\mathbb{N}$. From Eqn. (\ref{4.26}), we get the following relation
\begin{equation*}
\left| {_p\beta _k^{\left( n \right)}\left( {x + k} \right)} \right| = \frac{{n!p}}{{{x^{n + 1}}}} - \left| {_p\beta _k^{\left( n \right)}\left( x \right)} \right|.\label{4.46}\tag{4.46}    
\end{equation*}
\textbf{Proposition 4.2.} For $x>0$, $n\in\mathbb{N}$ and $p, k\in\mathbb{R}^{+}$, define
\begin{equation*}
{\Delta _n}\left( x \right) = \frac{{{x^{n + 1}}}}{{n!}}\left| {_p\beta _k^{\left( n \right)}\left( x \right)} \right|.\label{4.47}\tag{4.47}    
\end{equation*}
Then we have 
\begin{equation*}
\mathop {\lim }\limits_{x \to 0} {\Delta _n}\left( x \right) = p\label{4.48}\tag{4.48}    
\end{equation*}
and
\begin{equation*}
\mathop {\lim }\limits_{x \to 0} {{\Delta '}_n}\left( x \right) = 0.\label{4.49}\tag{4.49}    
\end{equation*}
\begin{proof}
From Eqn. (\ref{4.46}), we have
\begin{align*}
\mathop {\lim }\limits_{x \to 0} {\Delta _n}\left( x \right) &= \mathop {\lim }\limits_{x \to 0} \frac{{{x^{n + 1}}}}{{n!}}\left| {_p\beta _k^{\left( n \right)}\left( x \right)} \right|\label{4.50}\tag{4.50} \\&= \mathop {\lim }\limits_{x \to 0} \frac{{{x^{n + 1}}}}{{n!}}\left( {\frac{{n!p}}{{{x^{n + 1}}}} - \left| {_p\beta _k^{\left( n \right)}\left( {x + k} \right)} \right|} \right)\label{4.51}\tag{4.51}\\&= \mathop {\lim }\limits_{x \to 0} \left( {p - \frac{{{x^{n + 1}}}}{{n!}}\left| {_p\beta _k^{\left( n \right)}\left( {x + k} \right)} \right|} \right)\label{4.52}\tag{4.52} \\&= p.\label{4.53}\tag{4.53} 
\end{align*}
And,
\begin{align*}
\mathop {\lim }\limits_{x \to 0} {{\Delta '}_n}\left( x \right) &= \mathop {\lim }\limits_{x \to 0} \frac{d}{{dx}}\left( {\frac{{{x^{n + 1}}}}{{n!}}\left| {_p\beta _k^{\left( n \right)}\left( x \right)} \right|} \right)\label{4.54}\tag{4.54}  \\&= \mathop {\lim }\limits_{x \to 0} \left( {\frac{{{x^{n + 1}}}}{{n!}}\left| {_p\beta _k^{\left( {n + 1} \right)}\left( {x + k} \right)} \right| - \frac{{\left( {n + 1} \right){x^{n + 1}}}}{{n!}}\left| {_p\beta _k^{\left( n \right)}\left( {x + k} \right)} \right|} \right)\label{4.55}\tag{4.55} \\&= 0.\label{4.56}\tag{4.56} \end{align*}
This completes our proof.
\end{proof}
\begin{theorem}
For $n\in\mathbb{N}_{0}$, $r>0$, $s>0$, $\frac{1}{r}+\frac{1}{s}=1$ and $p, k\in\mathbb{R}^{+}$, we have 
\begin{equation*}
\left[ {\left| {_p\beta _k^{\left( n \right)}\left( {\frac{x}{r} + \frac{y}{s}} \right)} \right|} \right] \le {\left[ {\left| {_p\beta _k^{\left( n \right)}\left( x \right)} \right|} \right]^{\frac{1}{r}}}{\left[ {\left| {_p\beta _k^{\left( n \right)}\left( y \right)} \right|} \right]^{\frac{1}{s}}}.\label{4.57}\tag{4.57}    
\end{equation*}
\end{theorem}
\begin{proof}
Using Eqn. (\ref{4.24}) and Hölder's inequality, we have 
\begin{align*}
\left[ {\left| {_p\beta _k^{\left( n \right)}\left( {\frac{x}{r} + \frac{y}{s}} \right)} \right|} \right] &= \frac{p}{{{k^{n + 1}}}}\int\limits_0^\infty  {\frac{{{t^n}{e^{ - \left( {\frac{x}{{kr}} + \frac{y}{{ks}}} \right)t}}}}{{1 + {e^{ - t}}}}} dt \label{4.58}\tag{4.58}\\&= \frac{p}{{{k^{n + 1}}}}\int\limits_0^\infty  {\frac{{{t^{\frac{n}{r}}}{e^{ - \frac{{xt}}{{kr}}}}}}{{{{\left( {1 + {e^{ - t}}} \right)}^{\frac{1}{r}}}}}\frac{{{t^{\frac{n}{s}}}{e^{^{ - \frac{{yt}}{{ks}}}}}}}{{{{\left( {1 + {e^{ - t}}} \right)}^{\frac{1}{s}}}}}} dt \label{4.59}\tag{4.59}\\&\le {\left( {\frac{p}{{{k^{n + 1}}}}\int\limits_0^\infty  {\frac{{{t^n}{e^{ - \frac{{xt}}{k}}}}}{{\left( {1 + {e^{ - t}}} \right)}}} dt} \right)^{\frac{1}{r}}}{\left( {\frac{p}{{{k^{n + 1}}}}\int\limits_0^\infty  {\frac{{{t^n}{e^{ - \frac{{yt}}{k}}}}}{{\left( {1 + {e^{ - t}}} \right)}}} dt} \right)^{\frac{1}{s}}} \label{4.60}\tag{4.60}\\&= {\left[ {\left| {_p\beta _k^{\left( n \right)}\left( x \right)} \right|} \right]^{\frac{1}{r}}}{\left[ {\left| {_p\beta _k^{\left( n \right)}\left( y \right)} \right|} \right]^{\frac{1}{s}}}.\label{4.61}\tag{4.61}    
\end{align*}
This completes our proof. From this, follows the following theorem.
\end{proof}
\begin{theorem}
$\left| {_p\beta _k^{\left( n \right)}\left( x \right)} \right|$ is logarithmically convex for all $n\in\mathbb{N}$ on $\left( {0,\infty } \right)$.
\end{theorem}
\begin{theorem}
For $x, y>0$ and $p, k\in\mathbb{R}^{+}$, we have the following inequalities\newline
i)
\begin{equation*}
{\left( {\frac{{\left| {_p\beta _k^{\left( n \right)}\left( x \right)} \right|}}{{\left| {_p\beta _k^{\left( n \right)}\left( y \right)} \right|}}} \right)^{\frac{1}{{x - y}}}} \ge \exp \left( {\frac{{\left| {_p\beta _k^{\left( {n + 1} \right)}\left( y \right)} \right|}}{{\left| {_p\beta _k^{\left( n \right)}\left( y \right)} \right|}}} \right)\label{4.62}\tag{4.62}    
\end{equation*}
ii)
\begin{equation*}
\left| {_p\beta _k^{\left( {n + 2} \right)}\left( x \right)} \right|\left| {_p\beta _k^{\left( n \right)}\left( x \right)} \right| - {\left| {_p\beta _k^{\left( {n + 1} \right)}\left( x \right)} \right|^2} \ge 0.\label{4.63}\tag{4.63}    
\end{equation*}
\end{theorem}
\begin{proof}
The result follows from Eqn. (\ref{4.33}) and (\ref{4.34}).
\end{proof}
\textbf{Proposition 4.3.} For $n\in\mathbb{N}_{0}$, we have 
\begin{equation*}
{}_{p}\beta _k^{\left( n \right)}\left( x \right) = \frac{p}{k}\beta _k^{\left( n \right)}\left( x \right).\label{4.64}\tag{4.64}    
\end{equation*}
\begin{proof}
Using Eqn. (\ref{4.24}) and (\ref{A.1.6}), we have 
\begin{equation*}
{}_{p}\beta _k^{\left( n \right)}\left( x \right) = \frac{{{{\left( { - 1} \right)}^n}p}}{{{k^{n + 1}}}}\int\limits_0^\infty  {\frac{{{t^n}{e^{ - \frac{{xt}}{k}}}}}{{1 + {e^{ - t}}}}} dt = \frac{p}{k}\beta _k^{\left( n \right)}\left( x \right).\label{4.65}\tag{4.65}    
\end{equation*}
Thus
\begin{equation*}
{}_{p}\beta _k^{\left( n \right)}\left( x \right) = \frac{p}{k}\beta _k^{\left( n \right)}\left( x \right)\label{4.66}\tag{4.66}     
\end{equation*}
This completes our proof.
\end{proof}
It follows from the above proposition that 
\begin{equation*}
\left| {_p\beta _k^{\left( n \right)}\left( x \right)} \right| = \frac{p}{k}\left| {\beta _k^{\left( n \right)}\left( x \right)} \right|\label{4.68}\tag{4.68}    
\end{equation*}
where
\begin{equation*}
\left| {\beta _k^{\left( n \right)}\left( x \right)} \right| = {\left( { - 1} \right)^n}\beta _k^{\left( n \right)}\left( x \right).\label{4.69}\tag{4.69}    
\end{equation*}
\begin{theorem}
For $n\in\mathbb{N}_{0}$, $x, y>0$ and $p, k\in\mathbb{R}^{+}$, we have 
\begin{equation*}
\left| {_p\beta _k^{\left( n \right)}\left( {x + y} \right)} \right| < \left| {_p\beta _k^{\left( n \right)}\left( x \right)} \right| + \left| {_p\beta _k^{\left( n \right)}\left( y \right)} \right|\label{4.70}\tag{4.70}    
\end{equation*}
\end{theorem}
\begin{proof}
Multiply Eqn. (\ref{A.1.8}) with $\frac{p}{k}$ and use Eqn. (\ref{4.68}) to get the desire result.
\end{proof}
\begin{theorem}
Let $n\in\mathbb{N}_{0}$, $a>0$, and $x>0$, then the inequalities 
\begin{equation*}
\left| {{}_{p}\beta _k^{\left( n \right)}\left( {ax} \right)} \right| \le a\left| {{}_{p}\beta _k^{\left( n \right)}\left( x \right)} \right|\label{4.71}\tag{4.71}    
\end{equation*}
if $a\ge{1}$, and 
\begin{equation*}
\left| {{}_{p}\beta _k^{\left( n \right)}\left( {ax} \right)} \right| \ge a\left| {{}_{p}\beta _k^{\left( n \right)}\left( x \right)} \right|\label{4.72}\tag{4.72}    
\end{equation*}
if $a\le{1}$ are satisfied.
\end{theorem}
\begin{proof}
Multiply Eqn. (\ref{A.1.9}) and (\ref{A.1.10}) with $\frac{p}{k}$ and use Eqn. (\ref{4.68}) to get the desire result.
\end{proof}
\begin{theorem}
Let $k>0$ and $n\in\mathbb{N}_{0}$, then the inequality 
\begin{equation*}
\left| {{}_{p}\beta _k^{\left( n \right)}\left( {xy} \right)} \right| < \left| {{}_{p}\beta _k^{\left( n \right)}\left( x \right)} \right| + \left| {{}_{p}\beta _k^{\left( n \right)}\left( y \right)} \right|\label{4.73}\tag{4.73}    
\end{equation*}
holds for $x>0$ and $y\ge{1}$.
\end{theorem}
\begin{proof}
Multiply Eqn. (\ref{A.1.11}) with $\frac{p}{k}$ and use Eqn. (\ref{4.68}) to get the desire result.
\end{proof}
\subsection{On some multiplicative convex properties}
\begin{theorem}
For $x>0$, $n\in\mathbb{N}_{odd}$ and $p, k\in\mathbb{R}^{+}$, ${}_{p}\beta _k^{\left( n \right)}\left( x \right)$ is multiplicatively Convex on the interval $\left( {0,\infty } \right)$.
\end{theorem}
\begin{proof}
Using Lemma 2.3.4 (i) from \cite{21}, we can say that a function is strictly multiplicatively convex when it is logarithmically convex and increasing. Therefore using Eqn. (\ref{4.24}) and Hölder's inequality, we have 
\begin{align*}
{}_{p}\beta _k^{\left( n \right)}\left( {\frac{x}{r} + \frac{y}{s}} \right) &= \frac{{{{\left( { - 1} \right)}^n}p}}{{{k^{n + 1}}}}\int\limits_0^\infty  {\frac{{{t^n}{e^{ - \left( {\frac{x}{{kr}} + \frac{x}{{ks}}} \right)t}}}}{{1 + {e^{ - t}}}}dt}  \label{4.74}\tag{4.74}\\&= \frac{{{{\left( { - 1} \right)}^n}p}}{{{k^{n + 1}}}}\int\limits_0^\infty  {\frac{{{t^{\frac{n}{r}}}{e^{ - \frac{{xt}}{{kr}}}}}}{{{{\left( {1 + {e^{ - t}}} \right)}^{\frac{1}{r}}}}}\frac{{{t^{\frac{n}{s}}}{e^{ - \frac{{yt}}{{ks}}}}}}{{{{\left( {1 + {e^{ - t}}} \right)}^{\frac{1}{s}}}}}dt} \label{4.75}\tag{4.75} \\&\le {\left( {\frac{{{{\left( { - 1} \right)}^n}p}}{{{k^{n + 1}}}}\int\limits_0^\infty  {\frac{{{t^n}{e^{ - \frac{{xt}}{k}}}}}{{1 + {e^{ - t}}}}dt} } \right)^{\frac{1}{r}}}{\left( {\frac{{{{\left( { - 1} \right)}^n}p}}{{{k^{n + 1}}}}\int\limits_0^\infty  {\frac{{{t^n}{e^{ - \frac{{yt}}{k}}}}}{{1 + {e^{ - t}}}}dt} } \right)^{\frac{1}{s}}}\label{4.76}\tag{4.76} \\&= {\left[ {_p\beta _k^{\left( n \right)}\left( x \right)} \right]^{\frac{1}{r}}}{\left[ {_p\beta _k^{\left( n \right)}\left( y \right)} \right]^{\frac{1}{s}}}.\label{4.77}\tag{4.77}    
\end{align*}
Therefore, we can say that ${}_{p}\beta _k^{\left( n \right)}\left( x \right)$ is logarithmically convex. And, using Theorem 4.3, we can say that ${}_{p}\beta _k^{\left( n \right)}\left( x \right)$ is increasing when $n$ is odd. Therefore, we can conclude that ${}_{p}\beta _k^{\left( n \right)}\left( x \right)$ is multiplicatively convex for $n\in\mathbb{N}_{odd}$
\end{proof}
\begin{theorem}
Let $I$ be the interval $\left( {0,\infty } \right)$, $n\in\mathbb{N}_{odd}$ and $p, k\in\mathbb{R}^{+}-\left\{ 0 \right\}$, then for all ${x_1} \le {x_2} \le {x_3}$ in $I$, we have
\begin{equation*}
\left| {\begin{array}{*{20}{c}}
1&{\log {x_1}}&{\log \left( {_p\beta _k^{\left( n \right)}\left( {{x_1}} \right)} \right)}\\
1&{\log {x_2}}&{\log \left( {_p\beta _k^{\left( n \right)}\left( {{x_2}} \right)} \right)}\\
1&{\log {x_3}}&{\log \left( {_p\beta _k^{\left( n \right)}\left( {{x_3}} \right)} \right)}
\end{array}} \right| \ge 0 \label{4.78}\tag{4.78}   
\end{equation*}
or equivalently 
\begin{align*}
{}_{p}\beta _k^{\left( n \right)}{\left( {{x_1}} \right)^{\log {x_3}}}_p\beta _k^{\left( n \right)}{\left( {{x_2}} \right)^{\log {x_1}}}_p\beta _k^{\left( n \right)}{\left( {{x_3}} \right)^{\log {x_2}}}\ge & {}_{p}\beta _k^{\left( n \right)}{\left( {{x_1}} \right)^{\log {x_2}}}_p\beta _k^{\left( n \right)}{\left( {{x_2}} \right)^{\log {x_3}}}\\&\times {}_{p}\beta _k^{\left( n \right)}{\left( {{x_3}} \right)^{\log {x_1}}}.\label{4.79}\tag{4.79}       
\end{align*}
\end{theorem}
\begin{proof}
Using Theorem 4.15 and Lemma 2.3.1 from [\cite{21}, pg. 77], the desire result readily follows.
\end{proof}
\begin{theorem}
Let $I$ be the interval $\left( {0,\infty } \right)$, $n\in\mathbb{N}_{odd}$ and $p, k\in\mathbb{R}^{+}$. ${x_1} \ge {x_2} \ge .... \ge {x_n}$ and ${y_1} \ge {y_2} \ge .... \ge {y_n}$ are two families of numbers in a subinterval $I$ of $\left( {0,\infty } \right)$ such that
\begin{equation*}
{x_1} \ge {y_1}
\end{equation*}
\begin{equation*}
{x_1}{x_2} \ge {y_1}{y_2}    
\end{equation*}
\begin{equation*}
\begin{array}{*{20}{c}}
.\\
.\\
.
\end{array}        
\end{equation*}
\begin{equation*}
{x_1}{x_2}....{x_{n - 1}} \ge {y_1}{y_2}....{y_{n - 1}}    
\end{equation*}
\begin{equation*}
{x_1}{x_2}....{x_n} \ge {y_1}{y_2}....{y_n}.    
\end{equation*}
Then 
\begin{equation*}
{}_{p}\beta _k^{\left( n \right)}{\left( {{x_1}} \right)_p}\beta _k^{\left( n \right)}\left( {{x_2}} \right){...._p}\beta _k^{\left( n \right)}\left( {{x_n}} \right){ \ge _p}\beta _k^{\left( n \right)}{\left( {{y_1}} \right)_p}\beta _k^{\left( n \right)}\left( {{y_2}} \right){...._p}\beta _k^{\left( n \right)}\left( {{y_n}} \right).\label{4.80}\tag{4.80}    
\end{equation*}
\end{theorem}
\begin{proof}
Using Theorem 4.15 and Proposition 2.3.5 from from [\cite{21}, pg. 80], the desire result readily follows.
\end{proof}
\begin{theorem}
Let $I$ be the interval $\left( {0,\infty } \right)$, $m\in\mathbb{N}_{odd}$ and $p, k\in\mathbb{R}^{+}$.  Let $A \in {M_n}\left(  \mathbb{C}\right)$ be any matrix having the eigenvalues ${\lambda _1},....,{\lambda _n}$ and the singular numbers ${s_1},....,{s_n}$, listed such that $\left| {{\lambda _1}} \right| \ge .... \ge \left| {{\lambda _n}} \right|$ and ${s_1},....,{s_n}$. Then
\begin{equation*}
\prod\limits_{1 \le k \le n} {_p\beta _k^{\left( m \right)}\left( {{s_k}} \right) \ge \prod\limits_{1 \le k \le n} {_p\beta _k^{\left( m \right)}\left( {\left| {{\lambda _k}} \right|} \right)} }.\label{4.81}\tag{4.81}      
\end{equation*}
\end{theorem}
\begin{proof}
Using Theorem 4.15 and Proposition 2.3.6 from from [\cite{21}, pg. 80], the desire result readily follows.
\end{proof}
\textbf{Remark 4.1.} \textit{In a similar manner presented above, we can prove using the Hölder's inequality that both $\beta _k^{\left( n \right)}\left( x \right)$ and ${\beta ^{\left( n \right)}}\left( x \right)$ are multiplicatively convex on the interval $\left( {0,\infty } \right)$. And therefore, both will satisfy the above theorems.}\newline 

\subsection{On some monotonicity and convexity properties}
\begin{theorem}
If $F(x)$ is defined as 
\begin{equation*}
F\left( x \right) = {x^a}\left| {_p\beta _k^{\left( n \right)}\left( x \right)} \right|\label{4.82}\tag{4.82}
\end{equation*}
then, for $a\in\mathbb{R}$, $k\in\mathbb{R^{+}}$, $n\in\mathbb{N}_0$ and $x>0$, $F(x)$ is decreasing if $\frac{a}{k}\le n+1$ and increasing if $\frac{a}{k}\ge n+1+e^{-1}$. 
\end{theorem}
\begin{proof}
Using Eqn. (\ref{4.24}) and convolution theorem for Laplace transform, we have 
\begin{equation*}
F'\left( x \right) = a{x^{a - 1}}\left| {_p\beta _k^{\left( n \right)}\left( x \right)} \right| - {x^a}\left| {_p\beta _k^{\left( {n + 1} \right)}\left( x \right)} \right|\label{4.83}\tag{4.83}  
\end{equation*}
\begin{equation*}
\frac{{F'\left( x \right)}}{{{x^a}}} = \frac{a}{x}\left| {_p\beta _k^{\left( n \right)}\left( x \right)} \right| - \left| {_p\beta _k^{\left( {n + 1} \right)}\left( x \right)} \right|\label{4.84}\tag{4.84}  
\end{equation*}
\begin{equation*}
= \frac{a}{k}\frac{p}{{{k^{n + 1}}}}\int\limits_0^\infty  {{e^{ - \frac{x}{k}t}}} dt\int\limits_0^\infty  {\frac{{{t^n}{e^{ - \frac{x}{k}t}}}}{{1 + {e^{ - t}}}}} dt - \frac{p}{{{k^{n + 1}}}}\int\limits_0^\infty  {\frac{{{t^{n + 1}}{e^{ - \frac{x}{k}t}}}}{{1 + {e^{ - t}}}}} dt  \label{4.85}\tag{4.85}  
\end{equation*}
\begin{equation*}
= \frac{a}{k}\frac{p}{{{k^{n + 1}}}}\int\limits_0^\infty  {\left[ {\int\limits_0^t {\frac{{{s^n}}}{{1 + {e^{ - s}}}}ds} } \right]{e^{ - \frac{x}{k}t}}} dt - \frac{p}{{{k^{n + 1}}}}\int\limits_0^\infty  {\frac{{{t^{n + 1}}{e^{ - \frac{x}{k}t}}}}{{1 + {e^{ - t}}}}} dt \label{4.86}\tag{4.86}    
\end{equation*}
\begin{equation*}
= \frac{p}{{{k^{n + 1}}}}\int\limits_0^\infty  {{\xi _n}\left( t \right){e^{ - \frac{x}{k}t}}} dt, \label{4.87}\tag{4.87}   
\end{equation*}
where
\begin{equation*}
{\xi _n}\left( t \right) = \frac{a}{k}\int\limits_0^t {\frac{{{s^n}}}{{1 + {e^{ - s}}}}ds}  - \frac{{{t^{n + 1}}}}{{1 + {e^{ - t}}}}\label{4.88}\tag{4.88}    
\end{equation*}
Therefore, ${\xi _n}\left( 0 \right) = \mathop {\lim }\limits_{t \to {0^ + }} {\xi _n}\left( t \right) = 0$ and 
\begin{equation*}
{{\xi '}_n}\left( t \right) = \frac{a}{k}\frac{{{t^n}}}{{1 + {e^{ - t}}}} - \frac{{\left( {n + 1} \right){t^n}}}{{1 + {e^{ - t}}}} - \frac{{{t^{n + 1}}{e^{ - t}}}}{{{{\left( {1 + {e^{ - t}}} \right)}^2}}}  \label{4.89}\tag{4.89}  
\end{equation*}
\begin{equation*}
 = \frac{{{t^n}}}{{1 + {e^{ - t}}}}\left[ {\frac{a}{k} - \left( {n + 1} \right) - \frac{{t{e^{ - t}}}}{{1 + {e^{ - t}}}}} \right].\label{4.90}\tag{4.90}    
\end{equation*}
If $\frac{a}{k} \le n + 1$, then ${{\xi '}_n}\left( t \right)<0$ which implies that $F'\left( x \right)<0$, thus it gives the desire result. Similarly, we can prove that if $\frac{a}{k} \ge n + 1 + {e^{ - 1}}$ then ${{\xi '}_n}\left( t \right)>0$ which implies that $F'\left( x \right)>0$. This completes our proof.
\end{proof}
\begin{theorem}
Let $m\in\mathbb{N}$, the the inequality 
\begin{equation*}
\left| {_p\beta _k^{\left( m \right)}\left( {xy} \right)} \right| \le \left| {_p\beta _k^{\left( m \right)}\left( x \right)} \right| + \left| {_p\beta _k^{\left( m \right)}\left( y \right)} \right|\label{4.91}\tag{4.91} 
\end{equation*}
holds true for $x>0$ and $y\ge{1}$. 
\end{theorem}
\begin{proof}
Let 
\begin{equation*}
G\left( {x,y} \right) = \left| {_p\beta _k^{\left( m \right)}\left( {xy} \right)} \right| - \left| {_p\beta _k^{\left( m \right)}\left( x \right)} \right| - \left| {_p\beta _k^{\left( m \right)}\left( y \right)} \right|\label{4.92}\tag{4.92}    
\end{equation*}
Fix $y$ and differentiate with respect to $x$ to get
\begin{equation*}
\frac{\partial }{{\partial x}}G\left( {x,y} \right) =  - y\left| {_p\beta _k^{\left( m \right)}\left( {xy} \right)} \right| + \left| {_p\beta _k^{\left( {m + 1} \right)}\left( x \right)} \right|
\label{4.93}\tag{4.93}    
\end{equation*}
\begin{equation*}
= \frac{1}{x}\left[ {x\left| {_p\beta _k^{\left( {m + 1} \right)}\left( x \right)} \right| - xy\left| {_p\beta _k^{\left( m \right)}\left( {xy} \right)} \right|} \right].\label{4.94}\tag{4.94}    
\end{equation*}
From theorem 4.19, we know that $x\left| {_p\beta _k^{\left( m \right)}\left( x \right)} \right|$ is decreasing. Since $y\ge{1}$, this implies that $xy\ge{x}$, which therefore states that $G'\left( {x,y} \right)\ge{0}$ and thus is increasing. Then for $0 < x < \infty $, we have 
\begin{equation*}
G\left( {x,y} \right) \le \mathop {\lim }\limits_{x \to \infty } G\left( {x,y} \right) =  - \left| {_p\beta _k^{\left( m \right)}\left( y \right)} \right| < 0.\label{4.95}\tag{4.95}    
\end{equation*}
Putting the value of $G\left( {x,y} \right)$ in the above inequality yields the desire results.
\end{proof}
\begin{theorem}
Let $n\in\mathbb{N}$, then the function 
\begin{equation*}
H\left( x \right) = x\left| {_p\beta _k^{\left( n \right)}\left( x \right)} \right|\label{4.96}\tag{4.96}    
\end{equation*}
is strictly completely monotonic on $(0,\infty)$.
\end{theorem}
\begin{proof}
Differentiate Eqn. (\ref{4.96}) to get
\begin{equation*}
H'\left( x \right) = \left| {_p\beta _k^{\left( n \right)}\left( x \right)} \right| - x\left| {_p\beta _k^{\left( {n + 1} \right)}\left( x \right)} \right|
\label{4.97}\tag{4.97}    
\end{equation*}
\begin{equation*}
H''\left( x \right) =  - 2\left| {_p\beta _k^{\left( {n + 1} \right)}\left( x \right)} \right| + x\left| {_p\beta _k^{\left( {n + 2} \right)}\left( x \right)} \right|
\label{4.98}\tag{4.98}    
\end{equation*}
\begin{equation*}
H'''\left( x \right) = 3\left| {_p\beta _k^{\left( {n + 2} \right)}\left( x \right)} \right| - x\left| {_p\beta _k^{\left( {n + 3} \right)}\left( x \right)} \right|
\label{4.99}\tag{4.99}    
\end{equation*}
and therefore
\begin{equation*}
{H^{\left( m \right)}}\left( x \right) = {\left( { - 1} \right)^{m - 1}}m\left| {_p\beta _k^{\left( {m + n - 1} \right)}\left( x \right)} \right| + {\left( { - 1} \right)^m}x\left| {_p\beta _k^{\left( {m + n} \right)}\left( x \right)} \right|.
\label{4.100}\tag{4.100}    
\end{equation*}
Furthermore, we have
\begin{equation*}
\frac{{{{\left( { - 1} \right)}^m}{H^{\left( m \right)}}\left( x \right)}}{x} = \frac{{ - m}}{x}\left| {_p\beta _k^{\left( { n + m - 1} \right)}\left( x \right)} \right| + \left| {_p\beta _k^{\left( { n + m} \right)}\left( x \right)} \right|.
\label{4.101}\tag{4.101}    
\end{equation*}
Using the convolution theorem for Laplace transform, we have 
\begin{equation*}
= \frac{-m}{k}\frac{p}{{{k^{n + 1}}}}\int\limits_0^\infty  {{e^{ - \frac{x}{k}t}}} dt\int\limits_0^\infty  {\frac{{{t^{n+m-1}}{e^{ - \frac{x}{k}t}}}}{{1 + {e^{ - t}}}}} dt + \frac{p}{{{k^{n + 1}}}}\int\limits_0^\infty  {\frac{{{t^{n + m}}{e^{ - \frac{x}{k}t}}}}{{1 + {e^{ - t}}}}} dt  \label{4.102}\tag{4.102}  
\end{equation*}
\begin{equation*}
= \frac{-m}{k}\frac{p}{{{k^{n + 1}}}}\int\limits_0^\infty  {\left[ {\int\limits_0^t {\frac{{{s^{n+m-1}}}}{{1 + {e^{ - s}}}}ds} } \right]{e^{ - \frac{x}{k}t}}} dt + \frac{p}{{{k^{n + 1}}}}\int\limits_0^\infty  {\frac{{{t^{n + m}}{e^{ - \frac{x}{k}t}}}}{{1 + {e^{ - t}}}}} dt \label{4.103}\tag{4.103}    
\end{equation*}
\begin{equation*}
= \frac{p}{{{k^{n + 1}}}}\int\limits_0^\infty  {{\Omega _n}\left( t \right){e^{ - \frac{x}{k}t}}} dt, \label{4.104}\tag{4.104}   
\end{equation*}
where
\begin{equation*}
{\Omega _n}\left( t \right) = \frac{-m}{k}\int\limits_0^t {\frac{{{s^{n+m-1}}}}{{1 + {e^{ - s}}}}ds}  + \frac{{{t^{n + m}}}}{{1 + {e^{ - t}}}}\label{4.105}\tag{4.105}    
\end{equation*}
Therefore, ${\Omega _n}\left( 0 \right) = \mathop {\lim }\limits_{t \to {0^ + }} {\Omega _n}\left( t \right) = 0$ and 
\begin{equation*}
{{\Omega '}_n}\left( t \right) = \frac{-m}{k}\frac{{{t^{n+m-1}}}}{{1 + {e^{ - t}}}} + \frac{{\left( {n + m} \right){t^{n+m-1}}}}{{1 + {e^{ - t}}}} + \frac{{{t^{n + m}}{e^{ - t}}}}{{{{\left( {1 + {e^{ - t}}} \right)}^2}}}  \label{4.106}\tag{4.106}  
\end{equation*}
\begin{equation*}
 = \frac{{{t^{n+m-1}}}}{{1 + {e^{ - t}}}}\left[ {\frac{-m}{k} + \left( {n + m} \right) + \frac{{t{e^{ - t}}}}{{1 + {e^{ - t}}}}} \right]>0.\label{4.107}\tag{4.107}    
\end{equation*}
Hence ${\Omega _n}\left( t \right)$ is increasing. Therefore, for $t>0$, we have ${\Omega _n}\left( t \right)>{\Omega _n}\left( 0 \right)=0$ and thus ${{\left( { - 1} \right)}^m}{H^{\left( m \right)}}>0$. This competes our proof. 
\end{proof}

\section{On some extensions of Chaudhary-Zubair gamma function }
\subsection{\textit{p-k-}Chaudhary-Zubair gamma function}
The aim of this section is to provide  a $p$-$k$ extesion of the Chaudhary-Zubair gamma function \cite{22} defined as follows for $p, x>0$
\begin{equation*}
{\Gamma _p}\left( x \right) = \int\limits_0^\infty  {{t^{x - 1}}{e^{  \left( {-t - \frac{p}{t}} \right)}}} dt .\label{5.1}\tag{5.1}   
\end{equation*}
When $p=$, ${\Gamma _p}\left( x \right)$ reduces to ${\Gamma }\left( x \right)$. It satisfies the following properties
\begin{equation*}
{\Gamma _p}\left( {x + 1} \right) = x{\Gamma _p}\left( x \right) + p{\Gamma _p}\left( {x - 1} \right),\label{5.2}\tag{5.2}  \end{equation*}
\begin{equation*}
{\Gamma _p}\left( { - x} \right) = {p^{ - x}}{\Gamma _p}\left( x \right).\label{5.3}\tag{5.3}    
\end{equation*}
Differentiating Eqn. (\ref{5.1}) $n$ times yields
\begin{equation*}
\Gamma _p^{\left( n \right)}\left( x \right) = \int\limits_0^\infty  {{{\left( {\ln t} \right)}^n}{t^{x - 1}}{e^{  \left( {-t - \frac{p}{t}} \right)}}} dt\label{5.4}\tag{5.4}    
\end{equation*}
We now establish the following extension of \ref{5.1} for $x>0$ and $c, p, k\in\mathbb{R}^{+}-\left\{ 0 \right\}$. 
\begin{equation*}
{\Gamma _{CZ:\left( {c,p,k} \right)}}\left( x \right) = \int\limits_0^\infty  {{t^{x - 1}}{e^{ \left( {-\frac{{{t^k}}}{p} - \frac{c}{{\frac{{{t^k}}}{p}}}} \right)}}} dt.\label{5.5}\tag{5.5}    
\end{equation*}
Note that we have slightly changed the notation by replacing $p$ with $c$ in Eqn. (\ref{5.1}) to avoid getting confused it with the $p$ that will appear in our above extended definition. Differentiating Eqn. (\ref{2.2}) $n$ times yields
\begin{equation*}
\Gamma _{CZ:\left( {c,p,k} \right)}^{\left( n \right)}\left( x \right) = \int\limits_0^\infty  {{{\left( {\ln t} \right)}^n}{t^{x - 1}}{e^{  \left( -{\frac{{{t^k}}}{p} - \frac{c}{{\frac{{{t^k}}}{p}}}} \right)}}} dt.\label{5.6}\tag{5.6}    
\end{equation*}
In this section, we are going to explore some properties of $\Gamma _{CZ:\left( {c,p,k} \right)}^{\left( n \right)}\left( x \right)$ and ${\Gamma _{CZ:(c,p,k)}}\left( x \right)$. 
\subsection{Holder's inequalities for \textit{p-k-}Chaudhary-Zubair gamma function}
\begin{theorem}
For $x, y>0$, $\alpha, \beta \in (0,1)$, $\alpha + \beta =1$, $m,n \in \left\{2s :s\in\mathbb{N}_{0}  \right\} $ and $p, k\in\mathbb{R}^{+}-\left\{ 0 \right\}$, $\Gamma _{CZ:(c,p,k)}^{\left( n \right)}\left( x \right)$ satisfies the following inequality
\begin{equation*}
\Gamma _{CZ:(c,p,k)}^{\left( {\alpha m + \beta n} \right)}\left( {\alpha x + \beta y} \right) \le {\left[ {\Gamma _{CZ:(c,p,k)}^{\left( m \right)}\left( x \right)} \right]^\alpha }{\left[ {\Gamma _{CZ:(c,p,k)}^{\left( n \right)}\left( y \right)} \right]^\beta }.\label{5.7}\tag{5.7}    
\end{equation*}
\end{theorem}
\begin{proof}
Using Eqn. (\ref{5.6}), we have 
\begin{align*}
\Gamma _{CZ:(c,p,k)}^{\left( {\alpha m + \beta n} \right)}\left( {\alpha x + \beta y} \right) = \int\limits_0^\infty  {{{\left( {\ln t} \right)}^{\alpha m + \beta n}}{t^{\alpha x + \beta y - \left( {\alpha  + \beta } \right)}}{e^{\left( { - \frac{{{t^k}}}{p} - \frac{c}{{\frac{{{t^k}}}{p}}}} \right)\left( {\alpha  + \beta } \right)}}} dt\label{5.8}\tag{5.8} 
\end{align*}
\begin{equation*}
= \int\limits_0^\infty  {{{\left( {\ln t} \right)}^{\alpha m}}{{\left( {\ln t} \right)}^{\beta n}}{t^{\alpha \left( {x - 1} \right)}}{t^{\beta \left( {x - 1} \right)}}{e^{\alpha \left( { - \frac{{{t^k}}}{p} - \frac{c}{{\frac{{{t^k}}}{p}}}} \right)}}{e^{\beta \left( { - \frac{{{t^k}}}{p} - \frac{c}{{\frac{{{t^k}}}{p}}}} \right)}}} dt\label{5.9}\tag{5.9}    
\end{equation*}
\begin{equation*}
= \int\limits_0^\infty  {{{\left( {\ln t} \right)}^{\alpha m}}{t^{\alpha \left( {x - 1} \right)}}{e^{\alpha \left( { - \frac{{{t^k}}}{p} - \frac{c}{{\frac{{{t^k}}}{p}}}} \right)}}{{\left( {\ln t} \right)}^{\beta n}}{t^{\beta \left( {x - 1} \right)}}{e^{\beta \left( { - \frac{{{t^k}}}{p} - \frac{c}{{\frac{{{t^k}}}{p}}}} \right)}}} dt.\label{5.10}\tag{5.10}    
\end{equation*}
Now, using Hölder's inequality. we have
\begin{align*}
\int\limits_0^\infty  {{{\left( {\ln t} \right)}^{\alpha m}}{t^{\alpha \left( {x - 1} \right)}}{e^{\alpha \left( { - \frac{{{t^k}}}{p} - \frac{c}{{\frac{{{t^k}}}{p}}}} \right)}}{{\left( {\ln t} \right)}^{\beta n}}{t^{\beta \left( {x - 1} \right)}}{e^{\beta \left( { - \frac{{{t^k}}}{p} - \frac{c}{{\frac{{{t^k}}}{p}}}} \right)}}} dt\label{5.11}\tag{5.11} 
\end{align*}
\begin{equation*}
\le {\left[ {{{\int\limits_0^\infty  {\left[ {{{\left( {\ln t} \right)}^{\alpha m}}{t^{\alpha \left( {x - 1} \right)}}{e^{\alpha \left( { - \frac{{{t^k}}}{p} - \frac{c}{{\frac{{{t^k}}}{p}}}} \right)}}} \right]} }^{\frac{1}{\alpha }}}dt} \right]^\alpha }{\left[ {{{\int\limits_0^\infty  {\left[ {{{\left( {\ln t} \right)}^{\beta n}}{t^{\beta \left( {y - 1} \right)}}{e^{\beta \left( { - \frac{{{t^k}}}{p} - \frac{c}{{\frac{{{t^k}}}{p}}}} \right)}}} \right]} }^{\frac{1}{\beta }}}dt} \right]^\beta }\label{5.12}\tag{5.12}    
\end{equation*}
\begin{align*}
& = {\left[ {\int\limits_0^\infty  {{{\left( {\ln t} \right)}^m}{t^{\left( {x - 1} \right)}}{e^{\left( { - \frac{{{t^k}}}{p} - \frac{c}{{\frac{{{t^k}}}{p}}}} \right)}}} dt} \right]^\alpha }{\left[ {\int\limits_0^\infty  {{{\left( {\ln t} \right)}^n}{t^{\left( {y - 1} \right)}}{e^{\left( { - \frac{{{t^k}}}{p} - \frac{c}{{\frac{{{t^k}}}{p}}}} \right)}}} dt} \right]^\beta }\label{5.13}\tag{5.13}\\& = {\left[ {\Gamma _{CZ:(c,p,k)}^{\left( m \right)}\left( x \right)} \right]^\alpha }{\left[ {\Gamma _{CZ:(c,p,k)}^{\left( n \right)}\left( y \right)} \right]^\beta }. \label{5.14}\tag{5.14}   
\end{align*}
This completes our proof.
\end{proof}
\begin{corollary}
We have
\begin{equation*}
\Gamma _{CZ:(c,p,k)}^{\left( n \right)}\left( {\alpha x + \beta y} \right) \le {\left[ {\Gamma _{CZ:(c,p,k)}^{\left( n \right)}\left( x \right)} \right]^\alpha }{\left[ {\Gamma _{CZ:(c,p,k)}^{\left( n \right)}\left( y \right)} \right]^\beta }.\label{5.15}\tag{5.15}    
\end{equation*}
\end{corollary}
\begin{proof}
Let $m=n$ in Eqn. (\ref{5.7}) and the desire result readily follows.
\end{proof}
\begin{corollary}
We have
\begin{equation*}
\Gamma _{CZ:(c,p,k)}^{\left( {\frac{{m + n}}{2}} \right)}\left( {\frac{{x + y}}{2}} \right) \le \sqrt {\left[ {\Gamma _{CZ:(c,p,k)}^{\left( m \right)}\left( x \right)} \right]\left[ {\Gamma _{CZ:(c,p,k)}^{\left( n \right)}\left( y \right)} \right]}.\label{5.16}\tag{5.16}     
\end{equation*}
\end{corollary}
\begin{proof}
Let $\alpha = \beta = \frac{1}{2}$ in Eqn. (\ref{5.7}) and the desire result readily follows.
\end{proof}
\begin{corollary}
We have
\begin{equation*}
{\Gamma _{CZ:(c,p,k)}}\left( {\alpha x + \beta y} \right) \le {\left[ {{\Gamma _{CZ:(c,p,k)}}\left( x \right)} \right]^\alpha }{\left[ {{\Gamma _{CZ:(c,p,k)}}\left( y \right)} \right]^\beta }.\label{5.17}\tag{5.17}    
\end{equation*}
\end{corollary}
\begin{proof}
Let $m=n=0$ in Eqn. (\ref{5.7}) and the desire result readily follows.
\end{proof}
\begin{corollary}
$\Gamma _{CZ:(c,p,k)}^{\left( n \right)}\left( x \right)$ is logarithmically convex on the interval $(0,\infty)$ and increasing therefore, it is multiplicatively convex. 
\end{corollary}
\begin{theorem}
For $x, y>0$, $\alpha, \beta \in (0,1)$, $\alpha + \beta =1$ and $p, k\in\mathbb{R}^{+}-\left\{ 0 \right\}$, ${\Gamma _{CZ:(c,p,k)}}\left( x \right)$ satisfies the following inequality
\begin{equation*}
{\Gamma _{CZ:(c,p,k)}}\left( {x + y} \right) \le {\left[ {{\Gamma _{CZ:(c,p,k)}}\left( {\frac{x}{\alpha }} \right)} \right]^\alpha }{\left[ {{\Gamma _{CZ:(c,p,k)}}\left( {\frac{y}{\beta }} \right)} \right]^\beta }.\label{5.18}\tag{5.18}    
\end{equation*}
\end{theorem}
\begin{proof}
Using Eqn. (\ref{5.5}), we have
\begin{align*}
{\Gamma _{CZ:(c,p,k)}}\left( {x + y} \right) &= \int\limits_0^\infty  {{t^{\left( {x + y} \right) - 1}}{e^{\left( { - \frac{{{t^k}}}{p} - \frac{c}{{\frac{{{t^k}}}{p}}}} \right)}}} dt\label{5.19}\tag{5.19} \\&= \int\limits_0^\infty  {{t^{\left( {x + y} \right) - \left( {\alpha  + \beta } \right)}}{e^{\left( { - \frac{{{t^k}}}{p} - \frac{c}{{\frac{{{t^k}}}{p}}}} \right)\left( {\alpha  + \beta } \right)}}} dt\label{5.20}\tag{5.20} \\&= \int\limits_0^\infty  {{t^{x - \alpha }}{e^{\left( { - \frac{{{t^k}}}{p} - \frac{c}{{\frac{{{t^k}}}{p}}}} \right)\alpha }}} {t^{y - \beta }}{e^{\left( { - \frac{{{t^k}}}{p} - \frac{c}{{\frac{{{t^k}}}{p}}}} \right)\beta }}dt\label{5.21}\tag{5.21}   
\end{align*}
Using Hölder's inequality, we have 
\begin{align*}
&\int\limits_0^\infty  {{t^{x - \alpha }}{e^{\left( { - \frac{{{t^k}}}{p} - \frac{c}{{\frac{{{t^k}}}{p}}}} \right)\alpha }}} {t^{y - \beta }}{e^{\left( { - \frac{{{t^k}}}{p} - \frac{c}{{\frac{{{t^k}}}{p}}}} \right)\beta }}dt \\&\le {\left[ {{{\int\limits_0^\infty  {\left[ {{t^{x - \alpha }}{e^{\left( { - \frac{{{t^k}}}{p} - \frac{c}{{\frac{{{t^k}}}{p}}}} \right)\alpha }}} \right]} }^{\frac{1}{\alpha }}}dt} \right]^\alpha }{\left[ {{{\int\limits_0^\infty  {\left[ {{t^{y - \beta }}{e^{\left( { - \frac{{{t^k}}}{p} - \frac{c}{{\frac{{{t^k}}}{p}}}} \right)\beta }}} \right]} }^{\frac{1}{\beta }}}dt} \right]^\beta } \label{5.22}\tag{5.22}\\&= {\left[ {\int\limits_0^\infty  {{t^{\frac{x}{\alpha } - 1}}{e^{\left( { - \frac{{{t^k}}}{p} - \frac{c}{{\frac{{{t^k}}}{p}}}} \right)}}} dt} \right]^\alpha }{\left[ {\int\limits_0^\infty  {{t^{\frac{y}{\beta } - 1}}{e^{\left( { - \frac{{{t^k}}}{p} - \frac{c}{{\frac{{{t^k}}}{p}}}} \right)}}} dt} \right]^\beta }\label{5.23}\tag{5.23}\\& = {\left[ {{\Gamma _{CZ:(c,p,k)}}\left( {\frac{x}{\alpha }} \right)} \right]^\alpha }{\left[ {{\Gamma _{CZ:(c,p,k)}}\left( {\frac{y}{\beta }} \right)} \right]^\beta }.\label{5.24}\tag{5.24}    
\end{align*}
This completes our proof.
\end{proof}
\begin{corollary}
We have 
\begin{equation*}
{\Gamma _{CZ:(c,p,k)}}(x + y) \le \alpha {\Gamma _{CZ:(c,p,k)}}\left( {\frac{x}{\alpha }} \right) + \beta {\Gamma _{CZ:(c,p,k)}}\left( {\frac{x}{\beta }} \right)\label{5.25}\tag{5.25}    
\end{equation*}
\end{corollary}
\begin{proof}
Using Young's inequality (\ref{A.2.1}), the desire inequality readily follows.
\end{proof}
\begin{theorem} For $x, y>0$, $q\ge1$, $m,n \in \left\{2s :s\in\mathbb{N}_{0}  \right\} $ and $p, k\in\mathbb{R}^{+}-\left\{ 0 \right\}$, $\Gamma _{CZ:(c,p,k)}^{\left( n \right)}\left( x \right)$ satisfies the following inequality 
\begin{equation*}
{\left[ {\Gamma _{CZ:(c,p,k)}^{\left( m \right)}(x) + \Gamma _{CZ:(c,p,k)}^{\left( n \right)}(y)} \right]^{\frac{1}{q}}}={\left[ {\Gamma _{CZ:(c,p,k)}^{\left( m \right)}(x)} \right]^{\frac{1}{q}}} + {\left[ {\Gamma _{CZ:(c,p,k)}^{\left( n \right)}(y)} \right]^{\frac{1}{q}}}.\label{5.26}\tag{5.26}    
\end{equation*}
\end{theorem}
\begin{proof}
Using Eqn. (\ref{5.6}), we have
\begin{align*}
&{\left[ {\Gamma _{CZ:(c,p,k)}^{\left( m \right)}(x) + \Gamma _{CZ:(c,p,k)}^{\left( n \right)}(y)} \right]^{\frac{1}{q}}} \\&= {\left[ {\int\limits_0^\infty  {{{\left( {\ln t} \right)}^m}} {t^{x - 1}}{e^{\left( { - \frac{{{t^k}}}{k} - \frac{c}{{\frac{{{t^k}}}{k}}}} \right)}}dt + \int\limits_0^\infty  {{{\left( {\ln t} \right)}^n}} {t^{y - 1}}{e^{\left( { - \frac{{{t^k}}}{k} - \frac{c}{{\frac{{{t^k}}}{k}}}} \right)}}dt} \right]^{\frac{1}{q}}}\label{5.27}\tag{5.27} \\&= {\left[ {{{\int\limits_0^\infty  {\left[ {{{\left( {\ln t} \right)}^{\frac{m}{q}}}{t^{\frac{{x - 1}}{q}}}{e^{\frac{1}{q}\left( { - \frac{{{t^k}}}{k} - \frac{c}{{\frac{{{t^k}}}{k}}}} \right)}}} \right]} }^q} + {{\left[ {{{\left( {\ln t} \right)}^{\frac{n}{q}}}{t^{\frac{{y - 1}}{q}}}{e^{\frac{1}{q}\left( { - \frac{{{t^k}}}{k} - \frac{c}{{\frac{{{t^k}}}{k}}}} \right)}}} \right]}^q}dt} \right]^{\frac{1}{q}}}.\label{5.28}\tag{5.28}
\end{align*}
Using Minkowski's inequality (\ref{A.2.2}), we have 
\begin{align*}
&{\left[ {{{\int\limits_0^\infty  {\left[ {{{\left( {\ln t} \right)}^{\frac{m}{q}}}{t^{\frac{{x - 1}}{q}}}{e^{\frac{1}{q}\left( { - \frac{{{t^k}}}{k} - \frac{c}{{\frac{{{t^k}}}{k}}}} \right)}}} \right]} }^q} + {{\left[ {{{\left( {\ln t} \right)}^{\frac{n}{q}}}{t^{\frac{{y - 1}}{q}}}{e^{\frac{1}{q}\left( { - \frac{{{t^k}}}{k} - \frac{c}{{\frac{{{t^k}}}{q}}}} \right)}}} \right]}^q}dt} \right]^{\frac{1}{q}}}\\&\le {\left[ {{{\int\limits_0^\infty  {\left[ {{{\left( {\ln t} \right)}^{\frac{m}{q}}}{t^{\frac{{x - 1}}{q}}}{e^{\frac{1}{q}\left( { - \frac{{{t^k}}}{k} - \frac{c}{{\frac{{{t^k}}}{k}}}} \right)}}} \right]} }^q} + {{\left[ {{{\left( {\ln t} \right)}^{\frac{n}{q}}}{t^{\frac{{y - 1}}{q}}}{e^{\frac{1}{q}\left( { - \frac{{{t^k}}}{k} - \frac{c}{{\frac{{{t^k}}}{k}}}} \right)}}} \right]}^q}dt} \right]^{\frac{1}{q}}}\label{5.29}\tag{5.29} \\&\le {\left[ {\int\limits_0^\infty  {{{\left( {\ln t} \right)}^m}{t^{x - 1}}{e^{\left( { - \frac{{{t^k}}}{k} - \frac{c}{{\frac{{{t^k}}}{k}}}} \right)}}} dt} \right]^{\frac{1}{q}}} + {\left[ {\int\limits_0^\infty  {{{\left( {\ln t} \right)}^n}{t^{y - 1}}{e^{\left( { - \frac{{{t^k}}}{k} - \frac{c}{{\frac{{{t^k}}}{k}}}} \right)}}} dt} \right]^{\frac{1}{q}}}\label{5.30}\tag{5.30} \\&= {\left[ {\Gamma _{CZ:(c,p,k)}^{\left( m \right)}(x)} \right]^{\frac{1}{q}}} + {\left[ {\Gamma _{CZ:(c,p,k)}^{\left( n \right)}(y)} \right]^{\frac{1}{q}}}.\label{5.31}\tag{5.31}
\end{align*}
This completes our proof.
\end{proof}
\begin{theorem}
For $x>0$, $m,r \in \left\{2s :s\in\mathbb{N}_{0}  \right\} $ such that $m\ge{r}$ and $p, k\in\mathbb{R}^{+}-\left\{ 0 \right\}$, $\Gamma _{CZ:(c,p,k)}^{\left( n \right)}\left( x \right)$ satisfies the following inequality 
\begin{equation*}
\exp \left( {\Gamma _{CZ:(c,p,k)}^{\left( {m - r} \right)}\left( x \right)} \right)\exp \left( {\Gamma _{CZ:(c,p,k)}^{\left( {m + r} \right)}\left( x \right)} \right) \ge {\left( {\exp \left( {\Gamma _{CZ:(c,p,k)}^{\left( m \right)}\left( x \right)} \right)} \right)^2}.\label{5.32}\tag{5.32}    
\end{equation*}
\end{theorem}
\begin{proof}
Using Eqn. (\ref{5.6}), we have 
\begin{align*}
&\frac{1}{2}\left( {\Gamma _{CZ:(c,p,k)}^{\left( {m - r} \right)}\left( x \right) + \Gamma _{CZ:(c,p,k)}^{\left( {m + r} \right)}\left( x \right)} \right) - \Gamma _{CZ:(c,p,k)}^{\left( m \right)}\left( x \right) \\&= \frac{1}{2}\left( {\int\limits_0^\infty  {{{\left( {\ln t} \right)}^{m - r}}} {t^{x - 1}}{e^{\left( { - \frac{{{t^k}}}{k} - \frac{c}{{\frac{{{t^k}}}{k}}}} \right)}}dt + \int\limits_0^\infty  {{{\left( {\ln t} \right)}^{m + r}}} {t^{x - 1}}{e^{\left( { - \frac{{{t^k}}}{k} - \frac{c}{{\frac{{{t^k}}}{k}}}} \right)}}dt} \right) \\&- \int\limits_0^\infty  {{{\left( {\ln t} \right)}^m}} {t^{x - 1}}{e^{\left( { - \frac{{{t^k}}}{k} - \frac{c}{{\frac{{{t^k}}}{k}}}} \right)}}dt\label{5.33}\tag{5.33}\\&=\frac{1}{2}\int\limits_0^\infty  {\left[ {\frac{1}{{{{\left( {\ln t} \right)}^r}}} + {{\left( {\ln t} \right)}^r} - 2} \right]} {\left( {\ln t} \right)^m}{t^{x - 1}}{e^{\left( { - \frac{{{t^k}}}{k} - \frac{c}{{\frac{{{t^k}}}{k}}}} \right)}}dt\label{5.34}\tag{5.34}\\& = \frac{1}{2}\int\limits_0^\infty  {{{\left[ {1 - {{\left( {\ln t} \right)}^r}} \right]}^2}} {\left( {\ln t} \right)^{m - r}}{t^{x - 1}}{e^{\left( { - \frac{{{t^k}}}{k} - \frac{c}{{\frac{{{t^k}}}{k}}}} \right)}}\label{5.35}\tag{5.35} \\&\ge 0.\label{5.36}\tag{5.36}    
\end{align*}
Therefore,
\begin{equation*}
\frac{1}{2}\left( {\Gamma _{CZ:(c,p,k)}^{\left( {m - r} \right)}\left( x \right) + \Gamma _{CZ:(c,p,k)}^{\left( {m + r} \right)}\left( x \right)} \right) - \Gamma _{CZ:(c,p,k)}^{\left( m \right)}\left( x \right) \ge 0\label{5.37}\tag{5.37}     
\end{equation*}
\begin{equation*}
\Gamma _{CZ:(c,p,k)}^{\left( {m - r} \right)}\left( x \right) + \Gamma _{CZ:(c,p,k)}^{\left( {m + r} \right)}\left( x \right) \ge 2\Gamma _{CZ:(c,p,k)}^{\left( m \right)}\left( x \right).\label{5.38}\tag{5.38}     
\end{equation*}
Now, take the exponent of the above equation and the desire result readily follows.
\end{proof}

Similar inequalities for the original version of chaudhry-Zubair gamma function can be found in \cite{23}.\newline
\subsection{On some multiplicative convex properties}
Using corollary 5.1.4, we now know that $\Gamma _{CZ:(c,p,k)}^{\left( n \right)}\left( x \right)$ is multiplicatively convex. Therefore, it satisfies following theorems.
\begin{theorem}
Let $I$ be the interval $\left( {0,\infty } \right)$, $n \in \left\{2s :s\in\mathbb{N}_{0}  \right\} $ and $c, p, k\in\mathbb{R}^{+}-\left\{ 0 \right\}$, then for all ${x_1} \le {x_2} \le {x_3}$ in $I$, we have
\begin{equation*}
\left| {\begin{array}{*{20}{c}}
1&{\log {x_1}}&{\log \left( {\Gamma _{CZ:(c,p,k)}^{\left( n \right)}\left( {{x_1}} \right)} \right)}\\
1&{\log {x_2}}&{\log \left( {\Gamma _{CZ:(c,p,k)}^{\left( n \right)}\left( {{x_2}} \right)} \right)}\\
1&{\log {x_3}}&{\log \left( {\Gamma _{CZ:(c,p,k)}^{\left( n \right)}\left( {{x_3}} \right)} \right)}
\end{array}} \right| \ge 0\label{5.39}\tag{5.39}    
\end{equation*}
for all ${x_1} \le {x_2} \le {x_3}$ in $I$; equivalently, if and only if
\begin{align*}
\Gamma _{CZ:(c,p,k)}^{\left( n \right)}{\left( {{x_1}} \right)^{\log {x_3}}}&\Gamma _{CZ:(c,p,k)}^{\left( n \right)}{\left( {{x_2}} \right)^{\log {x_1}}}\Gamma _{CZ:(c,p,k)}^{\left( n \right)}{\left( {{x_3}} \right)^{\log {x_2}}} \\&\ge \Gamma _{CZ:(c,p,k)}^{\left( n \right)}{\left( {{x_1}} \right)^{\log {x_2}}}\Gamma _{CZ:(c,p,k)}^{\left( n \right)}{\left( {{x_2}} \right)^{\log {x_3}}}\Gamma _{CZ:(c,p,k)}^{\left( n \right)}{\left( {{x_3}} \right)^{\log {x_1}}}.\label{5.40}\tag{5.40}    
\end{align*}
\end{theorem}
\begin{proof}
Using Corollary 5.1.4 and from from [\cite{21}, pg. 77], the desire result readily follows.
\end{proof}
\begin{theorem}
Let $I$ be the interval $\left( {0,\infty } \right)$, $n \in \left\{2s :s\in\mathbb{N}_{0}  \right\} $  and $c, p, k\in\mathbb{R}^{+}-\left\{ 0 \right\}$. ${x_1} \ge {x_2} \ge .... \ge {x_n}$ and ${y_1} \ge {y_2} \ge .... \ge {y_n}$ are two families of numbers in a subinterval $I$ of $\left( {0,\infty } \right)$ such that
\begin{equation*}
{x_1} \ge {y_1}
\end{equation*}
\begin{equation*}
{x_1}{x_2} \ge {y_1}{y_2}    
\end{equation*}
\begin{equation*}
\begin{array}{*{20}{c}}
.\\
.\\
.
\end{array}        
\end{equation*}
\begin{equation*}
{x_1}{x_2}....{x_{n - 1}} \ge {y_1}{y_2}....{y_{n - 1}}    
\end{equation*}
\begin{equation*}
{x_1}{x_2}....{x_n} \ge {y_1}{y_2}....{y_n}.    
\end{equation*}
Then 
\begin{align*}
\Gamma _{CZ:(c,p,k)}^{\left( n \right)}\left( {{x_1}} \right)\Gamma _{CZ:(c,p,k)}^{\left( n \right)}&\left( {{x_2}} \right)....\Gamma _{CZ:(c,p,k)}^{\left( n \right)}\left( {{x_3}} \right) \\&\ge \Gamma _{CZ:(c,p,k)}^{\left( n \right)}\left( {{y_1}} \right)\Gamma _{CZ:(c,p,k)}^{\left( n \right)}\left( {{y_2}} \right)....\Gamma _{CZ:(c,p,k)}^{\left( n \right)}\left( {{y_n}} \right).\label{5.41}\tag{5.41}        
\end{align*}
\end{theorem}
\begin{proof}
Using Corollary 5.1.4 and from from [\cite{21}, pg. 80] the desire result readily follows.
\end{proof}
\begin{theorem}
Let $I$ be the interval $\left( {0,\infty } \right)$, $m \in \left\{2s :s\in\mathbb{N}_{0}  \right\} $ and $c, p, k\in\mathbb{R}^{+}-\left\{ 0 \right\}$.  Let $A \in {M_n}\left(  \mathbb{C}\right)$ be any matrix having the eigenvalues ${\lambda _1},....,{\lambda _n}$ and the singular numbers ${s_1},....,{s_n}$, listed such that $\left| {{\lambda _1}} \right| \ge .... \ge \left| {{\lambda _n}} \right|$ and ${s_1},....,{s_n}$. Then
\begin{equation*}
\prod\limits_{1 \le k \le n} {\Gamma _{CZ:(c,p,k)}^{\left( m \right)}\left( {{s_k}} \right)}  \ge \prod\limits_{1 \le k \le n} {\Gamma _{CZ:(c,p,k)}^{\left( m \right)}\left( {\left| {{\lambda _k}} \right|} \right)}.\label{5.42}\tag{5.42}      
\end{equation*}
\end{theorem}
\begin{proof}
Using Corollary 5.1.4 and from [\cite{21}, pg. 80], the desire result readily follows.
\end{proof}
\section{Ighachanea-Akkouchia Holder's inequalities for \textit{p-k-} analogue of Nielsen's beta function}
In section 3, we defined the \textit{p-k-}analogue of the Nielsen's beta function as 
\begin{equation*}
{}_{p}{\beta _k}\left( x \right) = \frac{p}{k}\int\limits_0^1 {\frac{{{t^{\frac{x}{k} - 1}}}}{{1 + t}}dt}\label{6.1}\tag{6.1}.    
\end{equation*}
For the sake of this section, we change the subscripts of ${}_{p}{\beta _k}\left( x \right)$ from $p$ and $k$ to $u$ and $v$:
\begin{equation*}
{}_{u}{\beta _v}\left( x \right) = \frac{u}{v}\int\limits_0^1 {\frac{{{t^{\frac{x}{v} - 1}}}}{{1 + t}}dt}\label{6.2}\tag{6.2}.    
\end{equation*}

\begin{theorem}
	Let $p_{k} >1$ for $k=1,2,\ldots,n$ with $\sum_{k=1}^{n}\frac{1}{p_{k}}=1$ and $x_{k} \geq0$. Then for all integers $m\geq2$, we have
	
	\begin{eqnarray*}
{}_{u}\beta_{v}\Big{(}\sum_{k=1}^{n}\frac{1}{p_{k}}x_{k}\Big{)}&+&nr_{0}^{m}\prod_{k=1}^{n}{}_{u}\beta_{v}^{\frac{1}{p_{k}}}(x_{k})\Big{(}1-\prod_{k=1}^{n}{}_{u}\beta_{v}^{\frac{-1}{n}}(x_{k}){}_{u}\beta_{v}\Big{(}\sum_{k=1}^{n}\frac{1}{n}x_{k}\Big{)}\Big{)}\\
		&&\hspace{-2.5cm}\leq\sum_{(i_1,\ldots,i_{n-1})\in A} C_{A} \frac{1}{p_{1}^{i_{0}-i_{1}}..p_{n}^{i_{n-1}-i_{n}}}\prod_{k=1}^{n}{}_{u}\beta_{v}(x_{k})^{1-\frac{i_{k}-i_{k-1}}{m}}{}_{u}\beta_{v}\Big{(}\sum_{k=1}^{n}\frac{(i_{k}-i_{k-1})x_{k}}{m}\Big{)}\leq\prod_{k=1}^{n}{}_{u}\beta_{v}^{\frac{1}{p_{k}}}(x_{k}),
	\end{eqnarray*}
	where, $r_{0}=\min\{\frac{1}{p_{k}},\;k=1,\ldots,n\}.$	
\end{theorem}

\par

\begin{proof}
	To apply Theorem 3.1, we set $\Omega:=(0,1)$ and take the measure 
	$d\mu(t):=\frac{1}{t(t+1)}dt$. Then we 
	choose $\xi_{k}(t)=\frac{u}{v}t^{\frac{x_{k}}{vp_{k}}},$ for $k=1,2,\ldots,n.$  So we have the following equalities: 
	\begin{align*}
	\int_{\Omega}^{}\prod_{k=1}^{n}|\xi_{k}(t)|d\mu(t)=	{}_{u}\beta_{v}\Big{(}\sum_{k=1}^{n}\frac{1}{p_{k}}x_{k}\Big{)},\label{6.3}\tag{6.3}
	\end{align*}
	\begin{align*}
	\int_{\Omega}^{}\prod_{k=1}^{n}|\xi_{k}(t)|^{\frac{p_{k}}{n}}d\mu(t)={}_{u}\beta_{v}\Big{(}\sum_{k=1}^{n}\frac{1}{n}x_{k}\Big{)},\label{6.4}\tag{6.4}
	\end{align*}
	\begin{align*}
	 ||\xi_{k}||_{p_{k}}=\Big{[}{}_{u}\beta_{v}{(x_{k})}\Big{]}^{1/p_{k}},\label{6.5}\tag{6.5}   
	\end{align*}
	and
	\begin{align*}
	\int_{\Omega}^{}\prod_{k=1}^{n}|\xi_{k}(t)|^{\frac{p_{k}(i_{k}-i_{k-1})}{m}}d\mu(t)={}_{u}\beta_{v}\Big{(}\sum_{k=1}^{n}\frac{(i_{k}-i_{k-1})x_{k}}{m}\Big{)}.\label{6.6}\tag{6.6}
	\end{align*}
	
	Now, using Theorem 3.1, we have
	
	\begin{eqnarray*}
		{}_{u}\beta_{v}\Big{(}\sum_{k=1}^{n}\frac{1}{p_{k}}x_{k}\Big{)}&+&nr_{0}^{m}\prod_{k=1}^{n}{}_{u}\beta_{v}^{\frac{1}{p_{k}}}(x_{k})\Big{(}1-\prod_{k=1}^{n}{}_{u}\beta_{v}^{\frac{-1}{n}}(x_{k}){}_{u}\beta_{v}\Big{(}\sum_{k=1}^{n}\frac{1}{n}x_{k}\Big{)}\Big{)}\\
		&&\hspace{-2.5cm}\leq\sum_{(i_1,\ldots,i_{n-1})\in A} C_{A} \frac{1}{p_{1}^{i_{0}-i_{1}}..p_{n}^{i_{n-1}-i_{n}}}\prod_{k=1}^{n}{}_{u}\beta_{v}(x_{k})^{1-\frac{i_{k}-i_{k-1}}{m}}{}_{u}\beta_{v}\Big{(}\sum_{k=1}^{n}\frac{(i_{k}-i_{k-1})x_{k}}{m}\Big{)}\leq\prod_{k=1}^{n}{}_{u}\beta_{v}^{\frac{1}{p_{k}}}(x_{k}).
	\end{eqnarray*}
	This completes our proof.
\end{proof}
Recall the n$^{th}$ derivative of Eq. (\ref{6.2}) in section 3, we have
\begin{align*}
{}_{u}\beta _v^{\left( N \right)}\left( x \right)  = \frac{{{{\left( { - 1} \right)}^N}u}}{{{v^{N + 1}}}}\int\limits_0^\infty  {\frac{{{t^N}{e^{ - \frac{{xt}}{v}}}}}{{1 + {e^{ - t}}}}dt}\label{6.7}\tag{6.7}
\end{align*}
Notice that here again we had a slight change in the notation, we have denoted the order by $N$ instead of $n$. Now, applying theorem 3.1 on ${}_{u}\beta _v^{\left( N \right)}\left( x \right) $ gives the following theorem.
\begin{theorem}
	Let $p_{k} >1$ for $k=1,2,\ldots,n$ with $\sum_{k=1}^{n}\frac{1}{p_{k}}=1$ and $x_{k} \geq0$. Then for all integers $m\geq2$ and $N\geq1$, we have
	\begin{eqnarray*}
		\Big{|}{}_{u}\beta^{(N)}_{v}\Big{(}\sum_{k=1}^{n}\frac{1}{p_{k}}x_{k}\Big{)}\Big{|}&+&nr_{0}^{m}\prod_{k=1}^{n}\Big{|}{}_{u}\beta^{(N)}_{v}(x_{k})\Big{|}^{\frac{1}{p_{k}}}\Big{(}1-\prod_{k=1}^{n}\Big{|}{}_{u}\beta^{(N)}_{v}(x_{k})\Big{|}^{\frac{-1}{n}}\Big{|}{}_{u}\beta^{(N)}_{v}\Big{(}\sum_{k=1}^{n}\frac{1}{n}x_{k}\Big{)}\Big{|}\Big{)}
	\end{eqnarray*}	
\begin{equation*}
\leq\sum_{(i_1,\ldots,i_{n-1})\in A} C_{A} \frac{1}{p_{1}^{i_{0}-i_{1}}..p_{n}^{i_{n-1}-i_{n}}}\prod_{k=1}^{n}\Big{|}{}_{u}\beta^{(N)}_{v}(x_{k})\Big{|}^{1-\frac{i_{k}-i_{k-1}}{m}}\Big{|}{}_{u}\beta^{(N)}_{v}\Big{(}\sum_{k=1}^{n}\frac{(i_{k}-i_{k-1})x_{k}}{m}\Big{)}\Big{|}\leq\prod_{k=1}^{n}\Big{|}{}_{u}\beta^{(N)}_{v}(x_{k}\Big{|}^{\frac{1}{p_{k}}},\label{6.8}\tag{6.8}    
\end{equation*}
		where, $r_{0}=\min\{\frac{1}{p_{k}},\;k=1,\ldots,n\}.$	
\end{theorem}
\section{Ighachanea-Akkouchia Holder's inequalities for extended Chaudhary-Zubair gamma function}
\subsection{ For \textit{p-k-} extended Chaudhary-Zubair gamma function}
Recall the definition of extended Chaudhary-Zubair gamma function that we presented in section 5:
\begin{equation*}
{\Gamma _{CZ:\left( {c,u,v} \right)}}\left( x \right) = \int\limits_0^\infty  {{t^{x - 1}}{e^{ \left( {-\frac{{{t^v}}}{u} - \frac{c}{{\frac{{{t^v}}}{u}}}} \right)}}} dt.\label{7.1}\tag{7.1}    
\end{equation*}
Again here we have replaced $p$ and $k$ with $u$ and $v$. The $N^{th}$ derivative of \ref{7.1} is given by 
\begin{equation*}
\Gamma _{CZ:\left( {c,u,v} \right)}^{\left( N \right)}\left( x \right) = \int\limits_0^\infty  {{{\left( {\ln t} \right)}^N}{t^{x - 1}}{e^{  \left( -{\frac{{{t^v}}}{u} - \frac{c}{{\frac{{{t^v}}}{u}}}} \right)}}} dt.\label{7.2}\tag{7.2}    
\end{equation*}
\begin{theorem}
	Let $p_{k} >1$ for $k=1,2,\ldots,n$ with $\sum_{k=1}^{n}\frac{1}{p_{k}}=1$ and $x_{k} \geq0$. Let $u,v>0$ Then for all integers $m\geq2$, we have

\begin{align*}
	\Gamma _{CZ:\left( {c,u,v} \right)}\Big{(}\sum_{k=1}^{n}\frac{1}{p_{k}}x_{k}\Big{)}&nr_{0}^{m}\prod_{k=1}^{n}\Gamma _{CZ:\left( {c,u,v} \right)}^{\frac{1}{p_{k}}}(x_{k})\Big{(}1-\prod_{k=1}^{n}\Gamma _{CZ:\left( {c,u,v} \right)}^{\frac{-1}{n}}(x_{k})\Gamma _{CZ:\left( {c,u,v} \right)}\Big{(}\sum_{k=1}^{n}\frac{1}{n}x_{k}\Big{)}\Big{)}\\&
	\hspace{-2.5cm}\leq\sum_{(i_1,\ldots,i_{n-1})\in A} C_{A} \frac{1}{p_{1}^{i_{0}-i_{1}}..p_{n}^{i_{n-1}-i_{n}}}\prod_{k=1}^{n}\Gamma _{CZ:\left( {c,u,v} \right)}(x_{k})^{1-\frac{i_{k}-i_{k-1}}{m}}\Gamma _{CZ:\left( {c,u,v} \right)}\Big{(}\sum_{k=1}^{n}\frac{(i_{k}-i_{k-1})x_{k}}{m}\Big{)}\\&
	\hspace{-2.5cm}\leq\prod_{k=1}^{n}\Gamma _{CZ:\left( {c,u,v} \right)}^{\frac{1}{p_{k}}}(x_{k}),\label{7.3}\tag{7.3}
\end{align*}
where, $r_{0}=\min\{\frac{1}{p_{k}},\;k=1,\ldots,n\}.$	
\end{theorem}

\par

For \ref{7.2} we have the following theorem.
\begin{theorem}
	Let $p_{k} >1$ for $k=1,2,\ldots,n$ with $\sum_{k=1}^{n}\frac{1}{p_{k}}=1$ and $x_{k} \geq0$. Let $u,v>0$ Then for all integers $m\geq2$, we have

\begin{align*}
	\Gamma _{CZ:\left( {c,u,v} \right)}^{\left( N \right)}\Big{(}\sum_{k=1}^{n}\frac{1}{p_{k}}x_{k}\Big{)}&nr_{0}^{m}\prod_{k=1}^{n}\Gamma _{CZ:\left( {c,u,v} \right)}^{{\left( N \right)}\frac{1}{p_{k}}}(x_{k})\Big{(}1-\prod_{k=1}^{n}\Gamma _{CZ:\left( {c,u,v} \right)}^{{\left( N \right)}\frac{-1}{n}}(x_{k})\Gamma _{CZ:\left( {c,u,v} \right)}^{\left( N \right)}\Big{(}\sum_{k=1}^{n}\frac{1}{n}x_{k}\Big{)}\Big{)}\\&\hspace{-2.5cm}\leq\sum_{(i_1,\ldots,i_{n-1})\in A} C_{A} \frac{1}{p_{1}^{i_{0}-i_{1}}..p_{n}^{i_{n-1}-i_{n}}}\prod_{k=1}^{n}\Gamma _{CZ:\left( {c,u,v} \right)}^{\left( N \right)}(x_{k})^{1-\frac{i_{k}-i_{k-1}}{m}}\Gamma _{CZ:\left( {c,u,v} \right)}^{\left( N \right)}\Big{(}\sum_{k=1}^{n}\frac{(i_{k}-i_{k-1})x_{k}}{m}\Big{)}\\&
	\hspace{-2.5cm}\leq\prod_{k=1}^{n}\Gamma _{CZ:\left( {c,u,v} \right)}^{{\left( N \right)}\frac{1}{p_{k}}}(x_{k}),\label{7.4}\tag{7.4}
\end{align*}
where, $r_{0}=\min\{\frac{1}{p_{k}},\;k=1,\ldots,n\}.$	
\end{theorem}

\subsection{ For ordinary Chaudhary-Zubair gamma function}
If $u=v$ in theorems 7.1 and 7.2, then we get the corresponding inequalities for the ordinary Chaudhary-Zubair gamma function:
\begin{equation*}
{\Gamma _c}\left( x \right) = \int\limits_0^\infty  {{t^{x - 1}}{e^{  \left( {-t - \frac{c}{t}} \right)}}} dt .\label{7.5}\tag{7.5}   
\end{equation*}
\begin{theorem}
	Let $p_{k} >1$ for $k=1,2,\ldots,n$ with $\sum_{k=1}^{n}\frac{1}{p_{k}}=1$ and $x_{k} \geq0$. Let $u,v>0$ Then for all integers $m\geq2$, we have

\begin{align*}
	\Gamma _{c}\Big{(}\sum_{k=1}^{n}\frac{1}{p_{k}}x_{k}\Big{)}&nr_{0}^{m}\prod_{k=1}^{n}\Gamma _{c}^{\frac{1}{p_{k}}}(x_{k})\Big{(}1-\prod_{k=1}^{n}\Gamma _{c}^{\frac{-1}{n}}(x_{k})\Gamma _{c}\Big{(}\sum_{k=1}^{n}\frac{1}{n}x_{k}\Big{)}\Big{)}\\&\hspace{-2.5cm}\leq\sum_{(i_1,\ldots,i_{n-1})\in A} C_{A} \frac{1}{p_{1}^{i_{0}-i_{1}}..p_{n}^{i_{n-1}-i_{n}}}\prod_{k=1}^{n}\Gamma _{c}(x_{k})^{1-\frac{i_{k}-i_{k-1}}{m}}\Gamma _{c}\Big{(}\sum_{k=1}^{n}\frac{(i_{k}-i_{k-1})x_{k}}{m}\Big{)}\leq\prod_{k=1}^{n}\Gamma _{c}^{\frac{1}{p_{k}}}(x_{k}),\label{7.6}\tag{7.6}
\end{align*}
where, $r_{0}=\min\{\frac{1}{p_{k}},\;k=1,\ldots,n\}.$	
\end{theorem}
\begin{theorem}
	Let $p_{k} >1$ for $k=1,2,\ldots,n$ with $\sum_{k=1}^{n}\frac{1}{p_{k}}=1$ and $x_{k} \geq0$. Let $u,v>0$ Then for all integers $m\geq2$, we have

\begin{align*}
	\Gamma _{c}^{\left( N \right)}\Big{(}\sum_{k=1}^{n}\frac{1}{p_{k}}x_{k}\Big{)}&nr_{0}^{m}\prod_{k=1}^{n}\Gamma _{c}^{{\left( N \right)}\frac{1}{p_{k}}}(x_{k})\Big{(}1-\prod_{k=1}^{n}\Gamma _{c}^{{\left( N \right)}\frac{-1}{n}}(x_{k})\Gamma _{c}^{\left( N \right)}\Big{(}\sum_{k=1}^{n}\frac{1}{n}x_{k}\Big{)}\Big{)}\\&\hspace{-2.5cm}\leq\sum_{(i_1,\ldots,i_{n-1})\in A} C_{A} \frac{1}{p_{1}^{i_{0}-i_{1}}..p_{n}^{i_{n-1}-i_{n}}}\prod_{k=1}^{n}\Gamma _{c}^{\left( N \right)}(x_{k})^{1-\frac{i_{k}-i_{k-1}}{m}}\Gamma _{c}^{\left( N \right)}\Big{(}\sum_{k=1}^{n}\frac{(i_{k}-i_{k-1})x_{k}}{m}\Big{)}\leq\prod_{k=1}^{n}\Gamma _{c}^{{\left( N \right)}\frac{1}{p_{k}}}(x_{k}),\label{7.7}\tag{7.7}
\end{align*}
where, $r_{0}=\min\{\frac{1}{p_{k}},\;k=1,\ldots,n\}.$	
\end{theorem}
\subsection{ For \textit{v-}extended Chaudhary-Zubair gamma function}
In \cite{28}, authors derieved the following extension of the Chaudhary-Zubair gamma function:
\begin{equation*}
{\Gamma _{b,v}}\left( z \right) = \int\limits_0^\infty  {{t^{z - 1}}{e^{ - \frac{{{t^v}}}{v} - \frac{{{b^v}{t^{ - v}}}}{v}}}dt} \label{7.8}\tag{7.8}   
\end{equation*}
for $\Re z>0$ and $b\geq0$ and $v>0$. The $N^{tk}$ derivative of ${\Gamma _{b,v}}\left( z \right)$ can be given by
\begin{equation*}
\Gamma _{b,v}^{\left( N \right)}\left( z \right) = \int\limits_0^\infty  {{{\left( {\ln t} \right)}^N}{t^{z - 1}}{e^{ - \frac{{{t^v}}}{v} - \frac{{{b^v}{t^{ - v}}}}{v}}}dt} \label{7.9}\tag{7.9}   
\end{equation*}
Applying theorem 3.1 on ${\Gamma _{b,v}}\left( z \right)$ and $\Gamma _{b,v}^{\left( N \right)}\left( z \right)$ gives the following theorems.
\begin{theorem}
	Let $p_{k} >1$ for $k=1,2,\ldots,n$ with $\sum_{k=1}^{n}\frac{1}{p_{k}}=1$ and $x_{k} \geq0$. Let $u,v>0$ Then for all integers $n\geq2$, we have

\begin{align*}
	\Gamma _{b,v}\Big{(}\sum_{k=1}^{n}\frac{1}{p_{k}}x_{k}\Big{)}&nr_{0}^{m}\prod_{k=1}^{n}\Gamma _{b,v}^{\frac{1}{p_{k}}}(x_{k})\Big{(}1-\prod_{k=1}^{n}\Gamma _{b,v}^{\frac{-1}{n}}(x_{k})\Gamma _{b,v}\Big{(}\sum_{k=1}^{n}\frac{1}{n}x_{k}\Big{)}\Big{)}\\&\hspace{-2.5cm}\leq\sum_{(i_1,\ldots,i_{n-1})\in A} C_{A} \frac{1}{p_{1}^{i_{0}-i_{1}}..p_{n}^{i_{n-1}-i_{n}}}\prod_{k=1}^{n}\Gamma _{b,v}(x_{k})^{1-\frac{i_{k}-i_{k-1}}{m}}\Gamma _{b,v}\Big{(}\sum_{k=1}^{n}\frac{(i_{k}-i_{k-1})x_{k}}{m}\Big{)}\leq\prod_{k=1}^{n}\Gamma _{b,v}^{\frac{1}{p_{k}}}(x_{k}),\label{7.10}\tag{7.10}
\end{align*}
where, $r_{0}=\min\{\frac{1}{p_{k}},\;k=1,\ldots,n\}.$	
\end{theorem}

\par

\begin{theorem}
	Let $p_{k} >1$ for $k=1,2,\ldots,n$ with $\sum_{k=1}^{n}\frac{1}{p_{k}}=1$ and $x_{k} \geq0$. Let $u,v>0$ Then for all integers $m\geq2$, we have

\begin{align*}
	\Gamma _{b,v}^{\left( N \right)}\Big{(}\sum_{k=1}^{n}\frac{1}{p_{k}}x_{k}\Big{)}&nr_{0}^{m}\prod_{k=1}^{n}\Gamma _{b,v}^{{\left( N \right)}\frac{1}{p_{k}}}(x_{k})\Big{(}1-\prod_{k=1}^{n}\Gamma _{b,v}^{{\left( N \right)}\frac{-1}{n}}(x_{k})\Gamma _{b,v}^{\left( N \right)}\Big{(}\sum_{k=1}^{n}\frac{1}{n}x_{k}\Big{)}\Big{)}\\&\hspace{-2.5cm}\leq\sum_{(i_1,\ldots,i_{n-1})\in A} C_{A} \frac{1}{p_{1}^{i_{0}-i_{1}}..p_{n}^{i_{n-1}-i_{n}}}\prod_{k=1}^{n}\Gamma _{b,v}^{\left( N \right)}(x_{k})^{1-\frac{i_{k}-i_{k-1}}{m}}\Gamma _{b,v}^{\left( N \right)}\Big{(}\sum_{k=1}^{n}\frac{(i_{k}-i_{k-1})x_{k}}{m}\Big{)}\leq\prod_{k=1}^{n}\Gamma _{b,v}^{{\left( N \right)}\frac{1}{p_{k}}}(x_{k}),\label{7.11}\tag{7.11}
\end{align*}
where, $r_{0}=\min\{\frac{1}{p_{k}},\;k=1,\ldots,n\}.$	
\end{theorem}
\section{Conclusion}
In this paper, using the theory of $k-$special functions and some extended versions of the gamma function, we have derived some new number theoretic functions such as the Nielsen's beta function and the extended Chaudhary-Zubair gamma function. Some monotonicity properties of this functions are also proved and modified Holder's inequalities which were derived by Ighachanea and Akkouchia in their work are applied in deriving some inequalities for the functionsthat we have presented in this paper.

\section{Appendix 1}
\textbf{1.1.}\cite{20} For $x, k>0$
\begin{align*}
{\beta _k}\left( x \right) &= \frac{k}{2}\left\{ {{\psi _k}\left( {\frac{{x + k}}{2}} \right) - {\psi _k}\left( {\frac{x}{2}} \right)} \right\}\label{A.1.1}\tag{A.1.1} \\&= \sum\limits_{n = 0}^\infty  {\left( {\frac{k}{{2nk + x}} - \frac{k}{{2nk + x + k}}} \right)} \label{A.1.2}\tag{A.1.2}\\& = \int\limits_0^\infty  {\frac{{{e^{ - \frac{{xt}}{k}}}}}{{1 + {e^{ - t}}}}dt} \label{A.1.3}\tag{A.1.3}\\& = \int\limits_0^1 {\frac{{{t^{\frac{x}{k} - 1}}}}{{1 + t}}dt}.\label{A.1.4}\tag{A.1.4} \end{align*}
\textbf{1.2.}\cite{20} For $n\in\mathbb{N}_{0}$, we have
\begin{align*}
\beta _k^{\left( n \right)}\left( x \right) &= \frac{k}{{{2^{n + 1}}}}\left\{ {\psi _k^{\left( n \right)}\left( {\frac{{x + k}}{2}} \right) - \psi _k^{\left( n \right)}\left( {\frac{x}{2}} \right)} \right\}\label{A.1.5}\tag{A.1.5} \\&= \frac{{{{\left( { - 1} \right)}^n}}}{{{k^n}}}\int\limits_0^\infty  {\frac{{{t^n}{e^{ - \frac{{xt}}{k}}}}}{{1 + {e^{ - t}}}}dt}\label{A.1.6}\tag{A.1.6} \\& = \int\limits_0^1 {\frac{{{{\left( {\ln t} \right)}^n}{t^{\frac{x}{k} - 1}}}}{{1 + t}}dt}.\label{A.1.7}\tag{A.1.7}     
\end{align*}
\textbf{1.3.}\cite{20} For $n\in\mathbb{N}_{0}$ and $x, y>0$, the following inequality holds true
\begin{equation*}
\left| {\beta _k^{\left( n \right)}\left( {x + y} \right)} \right| < \left| {\beta _k^{\left( n \right)}\left( x \right)} \right| + \left| {\beta _k^{\left( n \right)}\left( y \right)} \right|\label{A.1.8}\tag{A.1.8}
\end{equation*}
\textbf{1.4.}\cite{20} Let $n\in\mathbb{N}_{0}$, $a>0$, and $x>0$, then the inequalities 
\begin{equation*}
\left| {\beta _k^{\left( n \right)}\left( {ax} \right)} \right| \le a\left| {\beta _k^{\left( n \right)}\left( x \right)} \right|\label{A.1.9}\tag{A.1.9}    
\end{equation*}
if $a\ge{1}$, and 
\begin{equation*}
\left| {\beta _k^{\left( n \right)}\left( {ax} \right)} \right| \ge a\left| {\beta _k^{\left( n \right)}\left( x \right)} \right|\label{A.1.10}\tag{A.1.10}    
\end{equation*}
if $a\le{1}$ are satisfied.\newline
\textbf{1.5.}\cite{20} Let $k>0$ and $n\in\mathbb{N}_{0}$, then the inequality 
\begin{equation*}
\left| {\beta _k^{\left( n \right)}\left( {xy} \right)} \right| < \left| {\beta _k^{\left( n \right)}\left( x \right)} \right| + \left| {\beta _k^{\left( n \right)}\left( y \right)} \right|\label{A.1.11}\tag{A.1.11}    
\end{equation*}
holds for $x>0$ and $y\ge{1}$.
\section{Appendix 2}
\textbf{2.1.} (Young's Inequality): If $u,v \ge 0$ and $\left( {\alpha ,\beta } \right) \in \left( {0,1} \right)$ sucn that $\alpha+\beta=1$, then the inequality 
\begin{equation*}
{u^\alpha }{v^\beta } \le \alpha u + \beta v \label{A.2.1}\tag{A.2.1} \end{equation*}
holds.\newline
\textbf{2.2.} (Minkowski's inequality): Let $u\ge{1}$. If $f(t)$ and $g(t)$ are continuous real-valued function on $[a,b]$, then inequality
\begin{equation*}
{\left( {\int\limits_a^b {{{\left| {f\left( t \right) + g\left( t \right)} \right|}^u}} dt} \right)^{\frac{1}{u}}} \le {\left( {\int\limits_a^b {{{\left| {f\left( t \right)} \right|}^u}} dt} \right)^{\frac{1}{u}}} + {\left( {\int\limits_a^b {{{\left| {g\left( t \right)} \right|}^u}} dt} \right)^{\frac{1}{u}}}\label{A.2.2}\tag{A.2.2}    
\end{equation*}
holds.

\end{document}